\documentclass[11pt]{article}
 \usepackage{setspace} 
\usepackage[margin=1.in]{geometry}
\usepackage{mathtools}
\usepackage{amssymb} 
\usepackage{hyperref}
\usepackage[normalem]{ulem}
\usepackage{enumitem}
\usepackage{amsthm}
\usepackage{amsopn}
\usepackage{bbm}
\usepackage[numbers]{natbib}

\usepackage{amsmath}
\usepackage{enumitem}
\usepackage{amsfonts}
\usepackage{amssymb}
\usepackage{graphicx} 
\usepackage{enumitem}
\DeclarePairedDelimiterX\Basics[1](){ #1}
\usepackage[utf8]{inputenc}
\usepackage{color}
\usepackage{caption}
\usepackage{subcaption}
\usepackage{algorithm,algorithmic}
\usepackage{enumitem}
\usepackage{pifont}

\usepackage{pifont}
\usepackage{tikz}

\usepackage[T1]{fontenc}
\usepackage{etoolbox}
\usepackage{soul}

\setlist[enumerate,1]{label={\roman*)}}
\usepackage{float}
\newtheorem{theorem}{Theorem}
\newtheorem{corollary}{Corollary}[theorem]
\newtheorem{lemma}[theorem]{Lemma}
\newtheorem{remark}[theorem]{Remark}
\newtheorem{definition}[theorem]{Definition}
\newtheorem{proposition}[theorem]{Proposition}
\numberwithin{equation}{section}
\newtheorem{assumption}{Assumption}

\newcommand{\innermid}{\nonscript\;\delimsize\vert\nonscript\;}
\newcommand{\activatebar}{%
  \begingroup\lccode`\~=`\|
  \lowercase{\endgroup\let~}\innermid 
  \mathcode`|=\string"8000
}

\makeatletter

\newlist{steps}{enumerate}{1}
\setlist[steps, 1]{label = Step \arabic*:}

\newcommand{\interior}[1]{%
  {\kern0pt#1}^{\mathrm{o}}%
}
\newcommand{\E}{\mathbb{E}}

\newcommand{\qi}{\underline{q}}

\newcommand{\qis}{\underline{q}}
\newcommand{\qiis}{\overline{q}}

\newcommand{\cI}{c}

\newcommand{\tsigma}{\widetilde{\sigma}}

\begin{document}
\title{Decision Making under Costly Sequential Information Acquisition:\\ the Paradigm of Reversible and Irreversible Decisions}

\author{Renyuan Xu
\thanks{Department of Management Science and Engineering, Stanford University, Stanford, USA. \textbf{Email:} renyuanxu@stanford.edu. R.X. is partially supported by the NSF CAREER Award DMS-2524465..}
\and
Thaleia Zariphopoulou \thanks{Departments of Mathematics and IROM, The University of Texas at Austin, Austin,
USA,  and the Oxford-Man Institute, University of Oxford, Oxford, UK. \textbf{Email:} zariphop@math.utexas.edu}
\and
Luhao Zhang\thanks{Department of Applied Mathematics and Statistics, Johns Hopkins University, Baltimore, USA. \textbf{Email:} luhao.zhang@jhu.edu}
\thanks{
This work was presented at the 2022 SIAM Annual Meeting, the 2022 INFORMS Annual Meeting, the Finance Department Seminar at the Boston University Questrom School of Business, and the Random System CDT Workshop at the University of Oxford, Women in Mathematical Finance 2023 (Rutgers University), the Workshop on Decision Making and Uncertainty at the Institute for Mathematical and Statistical Innovation (IMSI), Mathematical Finance Seminar at Columbia University, the 2024 Bachelier World Congress, and the First INFORMS Conference on Financial Engineering and FinTech. An earlier version of part of this work was first posted online on December 2023.
We extend our gratitude to the participants for their valuable comments and suggestions, especially Steve Kou, Alejandra Quintos, Christoph Reisinger, Moris Strub,  Hao Xing, and Xun Yu Zhou. 
The authors would also like to thank IMSI, University of Chicago, for its generous hospitality during the Spring 2022 and Spring 2023 long programs, during which most of this work was completed. 
}}

\date{December 31, 2025}
\maketitle

\begin{abstract}
Decision making in modern stochastic systems, including e-commerce platforms, financial markets and healthcare systems, has evolved into a multifaceted process that combines information acquisition and adaptive information sources. This paper initiates a study on such integrated settings, where these elements are not only fundamental but, also, interact in a complex and stochastically intertwined manner.

We introduce a relatively simple model, which, however, captures the involved novel elements. A decision maker (DM) may choose between an established product $A$ of known value and a new product $B$ whose value is unknown. In parallel, the DM observes signals about the unknown value of product $B$ and can, also, opt to exchange it for product $A$ if $B$ is initially chosen. Mathematically, the model gives rise to sequential optimal stopping problems with distinct informational regimes (before and after buying product $B$), differentiated by the initial, coarser signal and the subsequent, more accurate one. We analyze in detail the underlying problems using predominantly viscosity solution techniques, departing from the existing literature on information acquisition which is based on traditional optimal stopping arguments. 

More broadly, the modeling approach introduced herein offers a novel framework for developing more complex interactions among decisions, information sources and information costs in stochastic environments, through a sequence of nested obstacle problems.
\end{abstract}

\maketitle

\section{Introduction}
\label{sec:introduction}
The paper introduces a new modeling framework and initiates a study that integrates costly information acquisition, adaptive information sources, and sequential decision making. The motivation comes from numerous applications in which these elements not only play a fundamental role but also interact and influence each other in a rather complex, stochastically intertwined way. 

\paragraph{Motivation.} Given the recent surge in generative AI \cite{bandi2023power,brynjolfsson2023generative,fui2023generative}, our framework is particularly valuable for studying the adoption of new AI technologies and subscription to new data services when seeking possible improvements in workflows and decision-making processes \cite{feder1984acquisition,godoe2012understanding,heiman2020marketing,usai2021unveiling}. In particular,
with rapidly advancing but still immature AI technologies, institutions and firms interested in adopting these innovations may {\it not fully understand their potential and limitations} when integrated into existing systems. Therefore, it is essential for institutional decision-makers to thoroughly {\it collect  information and assess new technologies} before implementation, while ensuring the {\it flexibility} to \textit{revert} to previous systems if the new solutions do not enhance workflow efficiency or if significant incompatibilities emerge \cite{stenbacka1994strategic}. Moreover, it is equally important for AI technology developers to understand whether providing the option to revert to older technologies would encourage companies, corporations and institutions to be more open to experimenting with new innovations \cite{claudy2015consumer,kaur2020innovation}.

Such situations also frequently arise in today’s data-driven decision-making landscape. Companies in data-driven industries often face the challenge of choosing between standard and premium data services \cite{mariani2020exploring,pei2020survey}. A critical consideration is to {\it collect information and understand whether the marginal gains from premium data justify the additional costs}. For instance, Google may consider purchasing either standard user data or more granular user data from Reddit to train its algorithms for targeted advertising \cite{hu2021rise,li2021developers}. Before finalizing a purchase, Reddit could provide Google with a small sample of the granular data in batches. Based on the results, Google may then decide whether to proceed with the premium service or opt for the standard, familiar one. If the premium service does not meet expectations, {\it switching} to the standard service is, also, a viable option.

Other application domains requiring a multifaceted framework that integrates information acquisition, adaptive sources, and sequential decision-making include healthcare, supply chain management and e-commerce. In medical care, doctors continually adjust treatment plans for chronic conditions such as diabetes, guided by ongoing assessments of patient health and lifestyle changes \cite{yan2013home}. However, learning about each patient’s evolving health conditions requires effort and care, and informational sources do not remain static. In finance, investors and fund managers make sequential investment decisions influenced by evolving market trends, economic data, and realized returns while in parallel seeking information from, frequently, distinct sources \cite{han2013social,he2022endogenization,holland1998financial}. Supply chain managers adjust inventory and logistics strategies based on real-time sales trends and consumer demands, while they harvest information from newly coming sources \cite{fu2010endogenous,li2014transparency}. In retail, consumers routinely return products they use for a short amount of time while, in parallel, learning more about their quality as well as about the price and quality of comparable products \cite{huettner2019consumer,kakhbod2021heterogeneous}.  In the context of dynamic contract design,  the principal faces the challenge of dynamically monitoring labor activities to minimize adverse event occurrences \cite{chen2020optimal}.

The above examples and many more underscore the critical importance of information in making informed, adaptive decisions in contexts where complete information is not only unavailable, but its acquisition is costly and, furthermore, it changes regimes and accuracy while decisions are, in parallel, being intertemporally made. 

To our knowledge, the literature on models that \textit{combine} sequential decision making with dynamically changing costly information sources is not adequately developed. Indeed, the existing models predominantly consider both a single information source and a single decision occurring at an (optimal) time, at which the problem entirely terminates. Herein, we aim at considerably relaxing both these features by allowing, from the one hand, distinct costly information sources across time and, from the other, sequential decisions which have the flexibility to include reversal of previous ones. 

\paragraph{Modeling framework.} Herein, we build a general model primarily motivated by E-commerce, yet applicable to a variety of scenarios previously discussed. On E-commerce platforms, we often encounter situations where a new product of unknown quality is introduced to compete against existing ones of known quality. A buyer who is choosing between this new product and several well-known products could collect information on the former by reading reviews, and decide whether to purchase it or not. Meanwhile, she could also access information about the pricing of existing products on the platform. As most of the E-commerce platforms offer flexible return and exchange services (especially for new products), buyers could first purchase either the new product or a known, established one, after a preliminary investigation. If the buyer is unsatisfied with the initial purchase, she has the option to return the product and exchange it with the alternative (possibly with a return fee). 
Another example, that falls into this framework, is premium service subscriptions. Service providing entities, such as data services and news subscription companies, frequently introduce alternative plans as a strategy to expand their business operations and stimulate increased customer engagement. Existing clients are presented with an option to choose between these new services and their current subscriptions. Furthermore, there is an opportunity for these clients to revert to their original plan, albeit at a financial cost, after a trial period with the new plan. 

We introduce a relatively simple model, which however captures the novel elements we consider. In an infinite horizon setting, a decision maker (DM) may choose between two products, $A$ and $B$. Product $A$ is well established and its value is entirely known. Product $B$, on the other hand, has been only recently introduced, and its actual value is unknown. The latter is modeled as a binomial random variable $\Theta$. While the possible two values of product $B$ are known to the DM, she learns about $\Theta$ under a noisy signal $Y_t$ over a time period before making her initial decision, to buy $A$ or $B$. This noisy signal may be, for example, generated from posted reviews of product $B$ on an E-commerce platform. 

At the initial time, the DM is offered the optionality that if the DM chooses product $B$, she is offered the optionality to return it and switch to product $A$ at a later instant chosen by the user. This optionality is not however cost-free as the DM must pay a return fee. On the other hand during this trial period, she is now able to evaluate product $B$ more accurately through its use and, more broadly, through a new informational signal which is naturally more accurate than the initial one. This updated learning procedure continues either till the DM decides to return the product, pay the return fee, and terminate the problem or indefinitely, if the DM decides to keep product $B$. See Figure \ref{fig:flowchart} for the timeline of the decisions.
In both cases, learning occurs within a Bayesian framework, as in the classical filtering theory. 

\begin{figure}
    \centering
    \includegraphics[width=0.8\textwidth]{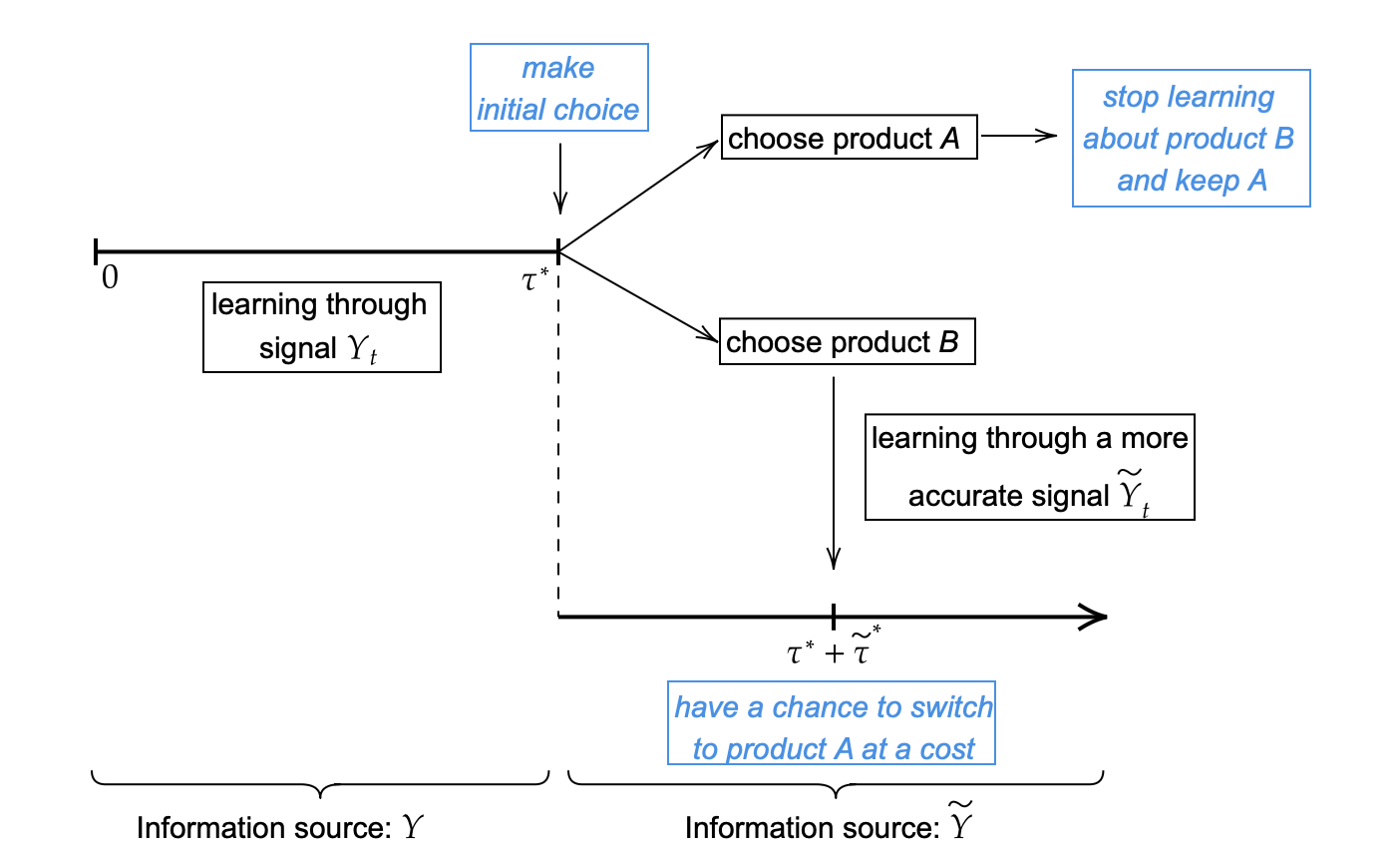}
    \caption{Timeline of the  decisions.} 
    \label{fig:flowchart}
\end{figure} 

Mathematically, the model gives rise to a sequential optimal stopping problem with two different informational regimes, differentiated by the initial, coarser signal and the subsequent, finer one.  More precisely, the obstacle term in the ``outer'' decision making problem is itself a value function of another, ``nested'', optimal stopping problem. Each of these two problems corresponds to distinct informational sources, costs, and utility flows and their analysis highlights the complex interactions between decisions, information sources, learning, and the associated payoffs and costs. 

\paragraph{Technical contribution.} We analyze the underlying problems primarily through viscosity solution techniques, as sufficient regularity is not always guaranteed. Furthermore, viscosity arguments allow for a unified approach to obtain existence, uniqueness, and comparison results. They are, also, flexible enough to perform sensitivity analysis and to study the limiting behavior of the solutions and the free boundaries, by building suitable viscosity sub- and super- solutions. 
Methodologically, our approach thus differs from traditional optimal stopping methods that rely on smooth fit or explicit free-boundary characterizations; see, for example, \cite{fudenberg2018speed,guo2001explicit,peskir2006optimal,van1976optimal}. Our viscosity techniques remain valid even when such conditions break down, offering a more flexible and generalizable analytical framework. In addition, their strong robustness and stability properties considerably facilitate the construction of numerical schemes when analytic solutions fail to exist.

Another key novelty of our work is the integration of reversible and irreversible decisions within a single optimal stopping framework, with a general, considerably more complex obstacle. Classical models typically treat irreversible decisions as stopping problems with fixed terminal conditions; see, for example, \cite{che2019optimal,fudenberg2018speed,moscarini2001optimal,zhong2017optimal}. Our formulation introduces a sequence of nested obstacle problems to capture interactions between sequential decisions, information sources, and acquisition costs, significantly extending standard models. The nature of this obstacle differs for each type of decision. The modeling approach herein offers a structural, both methodological and analytical, template to build new models and to analyze multi-stage decision-making with evolving costly information.

\paragraph{Related Literature.} Our work is related to two distinct lines of literature, one on decision making under \textit{costly information acquisition} and one on stochastic control with \textit{filtering}.

Decision making under information acquisition has a long history, dating back to the seminal paper \cite{wald1947foundations}, in which the flow of information is assumed to be fully exogenous. The DM only controls the decision time and action choice. Specifically, \cite{wald1947foundations} formulates an optimal stopping problem where the entire space of beliefs can be partitioned into a stopping region and a continuation region. Early works along this direction have focused on the duration of search when there exists a cost per unit of time when seeking information. For example, \cite{moscarini2001optimal} generalizes the framework in \cite{wald1947foundations} to an information intensity control problem where information is modeled as the trajectory of a Brownian motion with the drift representing the state, and the variance representing the intensity of information acquisition. Work \cite{miao2023dynamic} characterizes the solution to the dynamic rational inattention problem using a posterior-based approach, while \cite{ke2022parallel} tackles the problem of acquiring information on multiple alternatives simultaneously, characterizing a star-shaped optimal stopping boundary and the corresponding value function.

A similar setting is used in
\cite{fudenberg2018speed} to study the trade-off between information speed and precision. Some recent works have been trying to address the information selection issue when there are multiple information sources and when the DM has a limited capacity to process information. In this paradigm, \cite{che2019optimal,mayskaya2022dynamic} study the allocation of limited attention when there are multiple sources of information that are modeled by Poisson bandits. Work \cite{liang2017optimal} has a similar focus and assumes that the DM sequentially samples from a finite set of Gaussian signals and aims to predict a persistent multi-dimensional state at an unknown final period. The authors show that the optimal choice from Gaussian information sources is myopic. Recently, \cite{zhong2017optimal} studies a setting where the DM has access to both Gaussian signals (which are available in continuous time but are very noisy) and Poisson signals (which are less frequent but contain precise information). 
Another topical area of focus is partially observable Markov decision processes. 
 In this context, \cite{reisinger2022markov}  investigates the optimal times for acquiring costly observations and determining subsequent optimal action values. In addition, \cite{alizamir2022search} examines a scenario where the DM sequentially evaluates a stream of decision tasks to make optimal choices, exploring the impact of task accumulation and time pressure on her decisions.

However, despite the importance of addressing the search duration and information selection questions, decisions considered in these papers are either ``one shot'' (e.g., a static single decision between two products) or rather abstract (a general form of the terminal cost with no further analysis). As the  decision is an integral part of the framework and has a considerable impact on the learning process and information search behavior, it is crucial to discuss some more realistic, application oriented downstream decision making scenaria and examine how these tasks affect learning behavioral patterns. To the best of our knowledge, this aspect has been largely missing in the literature.

In the stochastic control literature with filtering, the DM has partial information about the underlying system, often modeled through a stochastic differential equation (SDE). The DM will first use the observation process to form an estimate of the state of the system and then, based on the separation principle \cite{shiryaev1973statistical}, construct the control signal as a function of this estimate. For this line of work, the observation process is often assumed to be obtained {\it at no cost}. The main focus, on the other hand, is on the construction of the filtering process and the solvability of the associated control problem  \cite{mitter1996filtering,sorenson1976overview}. Most of the studies have been focused on linear-quadratic problems \cite{anderson2007optimal,morris1976kalman,stengel1994optimal}, with a recent work \cite{knochenhauer2024continuous} incorporating costly information acquisition within this linear-quadratic model. There have, also, been recent works focusing on specific Bayesian problems, such as classification \cite{campbell2025bayesian} and sequential estimation \cite{campbell2025grab}, whereas our approach considers a more general value function. This is because a tractable finite dimensional Kalman-Bucy filter can be derived in explicit form when the underlying SDE is linear, and the associated control problem can be, in turn, solved through the Riccati system when the cost function is quadratic. Although sharing some common ingredients with our framework such as using the Bayesian formula to estimate unknown quantities, the partially observed quantities considered in this line of work are often more complex (i.e., unknown processes) than those considered in the literature on information acquisition (i.e., unknown variables). More importantly, the settings considered in stochastic control and filtering theory cannot be applied directly to the situation where the DM pays a cost to process information and to facilitate the understanding of how the cost of information affects the behavior of the DM. 

\paragraph{Our Contributions.} Our contribution is \textit{twofold}, modeling and methodological. We propose a general modeling approach that integrates sequential decision making, distinct informational sources, and information acquisition costs. To our knowledge, this is the first model with such features. It gives rise to a new kind of optimal stopping problems in which the obstacles themselves solve nested, distinct, optimal stopping problems. The obstacles take a general form which, in turn, requires a more involved analysis.

Methodologically, we develop a viscosity theory toolkit that allows us to study various aspects of the underlying optimal stopping problems. We also recover, as a special case, the well-known problem with a single (irreversible) decision which we solve with the complementary viscosity techniques. The viscosity theory allows us to perform a rather detailed analysis and study, among others, the effects of the return fee on both product-choice decisions between A and B, the behavior of the solution in terms of the volatility of the signals, the width of the exploratory regions in terms of the various modeling parameters, as well as various limiting cases. In particular, we find that i) the return optionality makes the DM spend less effort to explore the value of $\Theta$ during the first decision period, ii) if the exchange fee is comparatively small,  the DM makes her first choice earlier compared to the situation when the return cost is high, and iii) if the exchange cost is relatively high, the DM is more willing to choose the well-known product $A$ than choosing the new product $B$.

We finally note that the setting for the ``irreversible'' decisions part herein is akin to the set-up in \cite{moscarini2001optimal}, where the precision of the information signal can be further controlled at a cost, which is a special case in our general model. However, the mathematical tools used in our framework are markedly different. Work \cite{moscarini2001optimal} uses the smooth-fit principle to characterize the continuation and stopping regions under the assumption that the value function is twice continuously differentiable. In contrast, we construct proper viscosity sub- and super- solutions and utilize the comparison principle to identify the continuation and stopping regions, and to establish the asymptotic and monotonicity behavior with respect to the model parameters. Some of the constructions are by no means trivial and rather delicate, particularly in the sensitivity analysis part, due to interlinked dependence of model inputs. We find that our results complement the ones in \cite{moscarini2001optimal}, offering both broader insights and more technical details.

\medskip

The paper is organized as follows. In Section \ref{sec:framework}, we introduce the new extended model and study the underlying integrated optimal stopping problem. We present the main regularity results, the sensitivity analysis, and the limiting behavior of its solutions and of the free boundaries. In Section \ref{sec:single-decision}, we study the special case of irreversible decisions. In Section \ref{sec:reversible_decisions}, we analyze the general problem for two distinct kinds of information sources after the initial exploratory period. Namely, we study the cases of a Poisson and a Gaussian signal, respectively, and study various questions, among others, about the effects of the exchange fee on the optimal decisions of the DM. We conclude in Section \ref{sec:conclusion}.

\section{The general decision making model under costly sequential information }
\label{sec:framework}
A DM acts in an environment in which two products, $A$ and $%
B,$ are available. Product $A$ has a known value, modeled by a given
constant $\mu >0$, while product $B$ is new and not yet established. The quality of  product $B$
is not entirely known and is modeled by a random variable $\Theta,$ which may
take only two possible values, $l$ and $h$. The DM knows these values but
not the actual value of product $B$. For this, she formulates beliefs
using informational signals in a Bayesian framework. The signals, however,
are costly while they are being used. 

In all existing works, the DM\ starts at time $0,$ chooses a signal
process, and pays information acquisition costs while she learns about the new
product. At some time, say $\tau \geq 0,$ she chooses one of the two
products and the entire problem terminates. As mentioned in the Introduction, there is a rich literature on this
problem and we provide further comments in the sequel.

Herein, we depart from the known settings and propose a new, extended model
which allows for i)\ subsequent decision making beyond the initial (decision) time, ii) access to
sequentially differential information and iii)\ reversal of the initial
decision at a cost.

We describe the new model next, assuming for more simplicity that the initial time is $t=0$. 
Let $W = (W_t)_t$ be a  one-dimensional Brownian motion on a
complete probability space $(\Omega,\mathcal{F},\mathbb{F},\mathbb{P})$. 
 The DM observes a Gaussian signal process that follows 
\begin{equation}
dY_{t}=\,\Theta \,dt+\sigma dW_{t},\text{ \ }Y_{0}=0,  \label{Y-process}
\end{equation}%
with $\sigma >0$.
In turn, she
dynamically updates her views based on the information generated by $Y$,
acquiring the belief process $q$,
\begin{equation}
q_{t}:=\mathbb{P}\left[ \left. \Theta =h\right\vert \mathcal{F}_{t}^{Y}%
\right] ,  \label{belief-dfn}
\end{equation}%
with $\mathcal{F}_{t}^{Y}:=\sigma (Y_{s};0\leq s\leq t)$.  Let $\mathbb{F}^Y = (\mathcal{F}^Y_t)_t$ be the right-continuous extension
of the filtration generated by $Y$.

Classical results from filtering theory (see, for example, \cite{karatzas_zhao_2001}, also \cite{moscarini2001optimal,zhong2017optimal} and others)
yield that process $q$ is a martingale, solving 
\begin{equation}
dq_{t}=\frac{h-l}{\sigma }q_{t}\left( 1-q_{t}\right) dZ_{t},\text{ \ }%
q_{0}=q\in \left[ 0,1\right] ,  \label{belief-dynamics}
\end{equation}%
with $Z$ being a standard Brownian motion with respect to 
$\mathbb{F}^Y$. Clearly, the states $0$ and $1$ are absorbing, i.e. if the
initial belief $q_{0}=0,1$ then $q_{t}=0,1$, for $t>0$, respectively.

To have access to this signal process, the DM encounters \textit{information
acquisition cost}. It is assumed that they occur at a discounted rate $%
C(q)$ per unit of time, namely, the process (with a slight abuse of notation)
\begin{equation}
C_{t}:=\int_{0}^{t}e^{-\rho s}C(q_{s})ds,\text{ \ \ }\rho >0,
\label{information-cost}
\end{equation}%
is the cumulative cost of acquiring information in $\left[ 0,t\right] $
through signal $Y$. It is assumed that there are neither initial fixed costs
nor cost jumps thereafter. We comment on these features as well as on the sole dependence of the cost on $q$ in the last section.

\bigskip

\noindent\textbf{Examples}: i) \ $C(q)=C_I>0.$ The cumulative information cost is
deterministic with $C_{t}=\frac{C_I}{\rho }\left( 1-e^{-\rho t}\right) $. This constant cost case has been
extensively studied (see, \cite{keppo2008demand,keppo2022risky}), and we, also, study it herein
but in a more general context.

ii) $C\left( q\right) =\operatorname{Var}\left[ \left. \Theta \right\vert q\right] $ (or $%
C\left( q\right) =\sqrt{\operatorname{Var}\left[ \left. \Theta \right\vert q\right] }).$
The cumulative information cost is stochastic, $C_{t}=\int_{0}^{t}e^{-\rho
s}\operatorname{Var}\left[ \left. \Theta \right\vert q_{s}\right] ds$ (or $C_{t}=%
\int_{0}^{t}e^{-\rho s}\sqrt{\operatorname{Var}\left[ \left. \Theta \right\vert q_{s}\right]
}ds)$, and relates to the uncertainty in the DM's belief through time.\\

The value function of the DM, $V:\left[ 0,1\right] \rightarrow \mathbb{R},$
is defined as 
\begin{equation}
V(q)=\sup_{\tau \in \mathcal{T}^Y_0}\mathbb{E}\left. \left[ -\int_{0}^{\tau
}e^{-\rho t}C(q_{t})dt+e^{-\rho \tau }V_1(q_{\tau })\right\vert q_{0}=q%
\right] ,  \label{ValueFN-general}
\end{equation}%
where $\mathcal{T}^Y_t$ is the set of $\mathbb{F}^Y$-stopping times starting from time $t$:
\begin{equation}
\label{T-set}
    \mathcal T_t^Y \;:=\; \{\,\tau \ge t \;:\; \tau \text{ is an } \mathbb{F}^Y\text{- stopping time} \,\}.
\end{equation}

The function $V_1:\left[ 0,1\right] \rightarrow \mathbb{R}$ is modeled as  
\begin{equation}
V_1(q)=\max \Big(V_1^{A}( q),V_1^{B}(q)\Big),  \label{G-general}
\end{equation}%
where $V_1^{A}(q)$ and $V_1^{B}(q)$ represent the rewards the DM\ will receive 
{at times beyond the instant she makes her \textit{initial} choice. The
subscripts $A$ and $B$ correspond to the specific product \textit{initially} chosen by the DM. The DM evaluates the upcoming expected payoff $V_1$ and selects the greater of the two rewards, $V_1^A$ or $V_1^B$}.

Although the structure of \eqref{ValueFN-general}-\eqref{G-general} appears simple, note that $V_1^A$ and $V_1^B$ may represent the values of nested problems involving subsequent decisions, such as reversing the initial decision at a cost and possibly under a different information signal. Once the problems defining $V_1^A$ and $V_1^B$ are fully specified, solving the complete problem may still be nontrivial. In what follows, we present and study several cases for $V_1^A$ and $V_1^B$. A special case is introduced below in \eqref{eq:ir-v1}–\eqref{V-IR}, while a more general case is developed in Section~\ref{sec:nested}.

\paragraph{Special case: The existing single (irreversible) decision model.}

This is the only case that has been so far studied. There is a single decision and the problem terminates. The rewards are, in most works, given by   
\begin{equation}
V_1^{A}\left( q\right) =\mu \text{ \ \ and \ \ }V_1^{B}(q)=qh+(1-q)l,
\label{eq:ir-v1}
\end{equation}%
(see \cite{keppo2008demand,moscarini2001optimal,zhong2017optimal}). In other words, the payoff $V_1$ is the maximum of the
known return $\mu $ of product $A$ and the expected return $qh+(1-q)l$ of
product $B$ under belief $q$. The DM's value function takes the form
\begin{equation}
V(q)=\sup_{\tau \in {\mathcal{T}^Y_0}}\mathbb{E}\left. \left[ -\int_{0}^{\tau
}e^{-\rho t}C(q_{t})dt+e^{-\rho \tau }\max \Big( \mu ,q_{\tau
}h+(1-q_{\tau })l\Big) \right\vert q_{0}=q\right] .  \label{V-IR-generalG}
\end{equation}%
A popular case is when the information cost has a constant rate, $%
C(q)=C_{I}>0,$ $q\in \left[ 0,1\right],$ 

\begin{equation}
V(q):=\sup_{\tau \in {\mathcal{T}^Y_0}}\mathbb{E}\left. \left[ -\int_{0}^{\tau
}e^{-\rho t}C_{I}dt+e^{-\rho \tau }\max \Big( \mu ,q_{\tau }h+(1-q_{\tau
})l\Big) \right\vert q_{0}=q\right].  \label{V-IR}
\end{equation}
Some new properties of this single (irreversible) decision model, which is more general than those studied in the existing literature, are investigated in Section~\ref{sec:single-decision} and are of independent interest.

\subsection{The new general decision model with reversible decisions; nested optimal stopping problems}
\label{sec:nested}

We specify a general setting for $V^A$ and $V^B$, as introduced in \eqref{G-general}, as follows.
After the
initial decision (between $A$ or $B)$, the DM begins using the
chosen product, explores it and continues learning about it. For generality,
we assume that, even at this stage, complete information about the product
is not entirely accessible but, instead, a new informational signal is used which is, however, more accurate than the initial one. The DM\ has the optionality to return
the product at a cost and replace it with the other one, and then the problem terminates. Otherwise, she retains the product that was chosen initially. 

To keep the analysis tractable, we assume that the known product $A$ is not
returnable so the optionality to reverse the initial choice is available
only if the DM firstly chooses the unknown product $B.$ This is done only
for convenience as, mathematically, the analysis is the same, albeit more
involved. 

If the DM chooses product $B$ at $\tau^{\ast} $, she continues learning about it and updating the belief process via a \textit{new signal} process $ \widetilde{Y}_{t}$ $(t\ge \tau^*)$ on $(\Omega,\mathcal{F}_t,\mathbb{F},\mathbb{P})$, which is in analogy to (\ref{Y-process}). Here similar to \ref{Y-process},  $\widetilde Y_t$ is assumed to be a time homogeneous Markov process. The DM may decide to keep $B$ or exchange it for $A$, say at time $%
\tau^{\ast} +\widetilde{\tau}^{\ast}.$ If she returns it, she encounters a \textit{exchange fee%
} $R({q}_{{\tau}^{\ast}+\widetilde{\tau}^{\ast}})$ and the overall decision process
terminates. The functional form of fee $\widetilde R:{[0,1]\mapsto\mathbb{R}}$ is known to the DM at the initial time $0$. In general, $\widetilde{\tau}^*
\leq +\infty,$ thus allowing for the possibility of keeping product $B$ that was initially chosen $%
\left( \widetilde{\tau}^*=+\infty \right) $.  See Figure \ref{fig:reversible-B} for a demonstration.

In summary, \textit{if} she chooses product $B$ at time $\tau^{\ast}$, \textit{a new} optimization
problem is being generated on $\left[ \tau^{\ast} ,\infty \right) $ that
incorporates the optionality of exchange, the exchange fee and the
information acquisition under the new signal.  
The value function is given by
\begin{equation}
{\sup_{\widetilde{\tau}\in \widetilde{\mathcal{T}}_{\tau^*}}}\,%
\mathbb{{E}}\left[ \left. \int_{\tau^{\ast}}^{\tau^{\ast}+\widetilde{\tau }%
}e^{-\rho (t-\tau^{\ast})}\left( -\widetilde{C}\left( {q}_{t}\right)
+\widetilde m\left( {q}_{t}\right) )\right) dt+e^{-\rho \widetilde{\tau }
}\left( \mu -\widetilde R({q}_{\tau^{\ast}+\widetilde{\tau}})\right) \right| q_{\tau^\ast} 
\right] .  \label{U-general-pre}
\end{equation}
subject to 
\begin{eqnarray*}
    q_t=\mathbb{P}\Big[\Theta=h \,\,\Big|\,\,\widetilde {\mathcal{F}}_t\Big] =\mathbb{P}\Big[\Theta=h \,\,\Big|\,\,\sigma\big(\{\widetilde{Y}_s\}_{\tau^{\ast}\leq s\leq t}\big), q_{\tau^{\ast}}\Big] \quad \text{ for } \quad t \geq  \tau^{\ast},
\end{eqnarray*}
where $\widetilde{\mathcal{F}}_t =\sigma(\{Y_s\}_{s\leq  \tau^{\ast}},\{\widetilde{Y}_s\}_{ \tau^{\ast}\leq s\leq t})$ is the augmented information filtration for the DM. The second equality holds by the strong Markov property of the signal processes. Herein, $\widetilde{C}:[0,1]\mapsto \mathbb{R}_+$ denotes the information acquisition cost function for the refined signal  $\widetilde{Y}$ and $\widetilde{m}:[0,1]\mapsto \mathbb{R}_+$ denotes the payoff accumulated by the decision maker from using product  $B$ over the interval $\left[ \tau^{\ast}
,\tau^{\ast} +\widetilde{\tau}\right] $. Furthermore, in analogy to (\ref{T-set}), $
\widetilde{\mathcal{T}}_t$ is the set of stopping times adapted to $(\widetilde{\mathcal{F}}_s)_{s\ge t}$ for $t\ge \tau^\ast$.

By the strong Markov property and the time homogeneity of $\widetilde{Y}$,  for analytical convenience we shift the time index by $\tau^\ast$ and initialize the nested problem at time $0$. The resulting belief process $\widetilde{q}_t = \mathbb{P}[\Theta = h \mid \mathcal{F}^{\widetilde Y}_t]$ starts from $\widetilde{q}_0 = q_{\tau^\ast}$ (see \eqref{eq:belief-onlyPoisson} and \eqref{eq:belief-small-sigma}). Then, the value function in \eqref{U-general-pre} can be rewritten as
\begin{equation}
U^{B}(q_{\tau^{\ast}})\ :=\sup_{\widetilde{\tau}\in \mathcal{T}^{\widetilde{Y}}_0}\,%
\mathbb{E}\left[ \left. \int_{0}^{\widetilde{\tau }%
}e^{-\rho t}\left( -\widetilde{C}\left( \widetilde{q}_{t}\right)
+\widetilde m\left( \widetilde{q}_{t}\right) )\right) dt+e^{-\rho \widetilde{\tau }%
}\left( \mu -\widetilde R(\widetilde{q}_{\widetilde{\tau}})\right) \right\vert 
\widetilde{q}_{0}=q_{\tau^{\ast} }\right] ,  \label{U-general}
\end{equation}
where, similar to before, $\mathcal{T}^{\widetilde{Y}}_t$ contains all $\mathbb{F}^{\widetilde Y}$-stopping times starting from time $t$.
We will refer to this problem as the \textit{nested continuation problem for
product }$B$ and to $U^{B}$ as its \textit{nested continuation value
function}.

\begin{figure}[H]
    \centering    \includegraphics[width=0.8\textwidth]{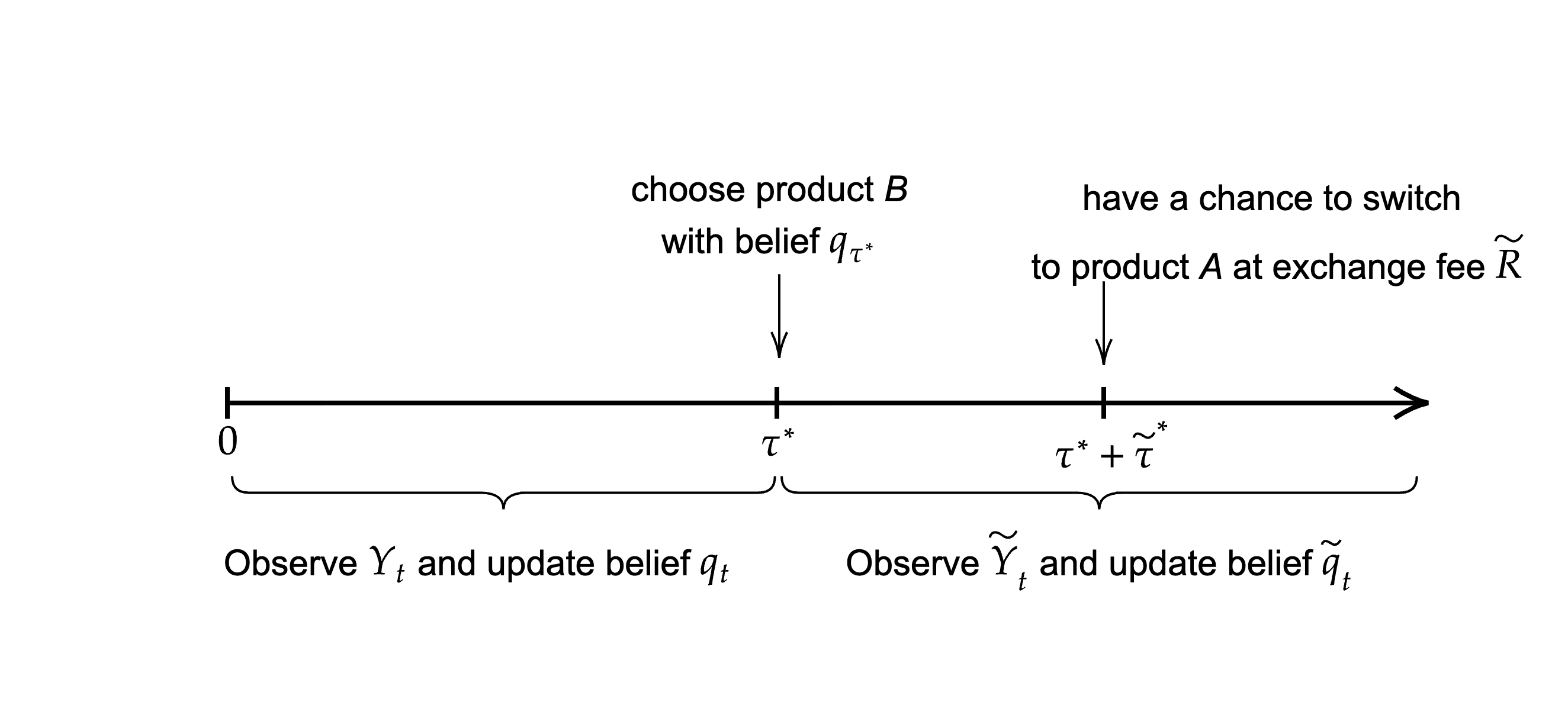}
    \caption{Timeline of the decisions when the initial choice is product $B$.} \label{fig:reversible-B}
\end{figure}

We are now ready to formulate the \textit{general, nested} model as follows. 
\begin{definition}[Nested Decision Making under Sequential Differential Information.]
\label{def:nested}
   The value function of the integrated model is defined as
\begin{equation}
V(q)=\sup_{\tau \in \mathcal{T}^{Y}_0}\mathbb{E}\left. \left[ -\int_{0}^{\tau
}e^{-\rho t}C(q_{t})dt+e^{-\rho \tau } V_1(q_\tau) \right\vert q_{0}=q\right],  \label{V}
\end{equation}%
with 
\begin{eqnarray}
    V_1(q):=\max(V_1^A(q),V_1^B(q)),
    \label{V1}
\end{eqnarray}
and
\begin{equation}
    V_1^A(q) = \mu \text{\ \  and \ \ } V_1^B(q) = U^B(q),
    \label{V1-general}
\end{equation}
where $U^B$ is defined in \eqref{U-general}. 
\end{definition}

The detailed mathematical foundations and associated stability results of the above framework are established in Section~\ref{sec:reversible_decisions}.

\subsection{General regularity results and sensitivity analysis}

Whether the DM chooses to work with the irreversible or the reversible
decision making problems, we will show that the mathematical analysis concerns optimal
stopping problems of the general form 
\begin{equation}
V(q)=\sup_{\tau \in \mathcal{T}^{Y}_0}\mathbb{E}\left. \left[ -\int_{0}^{\tau
}e^{-\rho t}C(q_{t})dt+e^{-\rho \tau }{G(q}_{\tau }{)}\right\vert q_{0}=q%
\right] ,\text{ \ }q\in \left[ 0,1\right] ,  \label{V-G}
\end{equation}
for \textit{general} information costs and payoff functions $C(.)$ and $G(.)
$,
with the belief process $q$ solving (cf. (\ref%
{belief-dynamics})) {up to time $\tau$},%
\begin{equation*}
dq_{t}=\frac{h-l}{\sigma }q_{t}\left( 1-q_{t}\right) dZ_{t},\text{ \ }%
q_{0}=q\in \left[ 0,1\right],
\end{equation*}%
 As mentioned earlier, only the specific case $(C(q),G(q))=(C_{I}$, ${\max (\mu
,qh+(1-q)l))}$ has been so far analyzed (see \cite{keppo2008demand}).

For times beyond $\tau$, similar optimal stopping problems arise, as \eqref{U-general} demonstrates. They differ only in the dynamics of the belief process in the new information regimes (see Poisson \eqref{eq:belief-onlyPoisson} and Gaussian \eqref{eq:belief-small-sigma}) and in their payoffs. However, their analysis is carried out using the detailed arguments we provide next for problem \eqref{V-G}, allowing for \textit{general} pairs $(C(.), G(.))$ of information costs and payoff functions, respectively. To this end, we first introduce some rather mild conditions.

\begin{assumption}
\label{ass:CI}
\begin{enumerate}
    \item[i)] The information cost $C:[0,1]\rightarrow \mathbb{R%
}_{+}$ is Lipschitz {with Lipschitz constant $L>0$, and for some }$%
C_{1},C_{2},$ $0<C_{1}\le C(q) \le C_{2},$ $q\in \left[ 0,1\right] .$
    \item [ii)] The payoff function $G:[0,1]\rightarrow \mathbb{R}_{+}$ is Lipschitz {%
with Lipschitz constant $K>0$, convex and non-decreasing.}
    \item [iii)] The model parameters satisfy
    \begin{equation}
    0<l<\mu <h\text{ \ \ and \  \ }\sigma >0.  \label{parameters}
    \end{equation}
\end{enumerate}
\end{assumption}

From classical results in optimal stopping, 
$V\ $is expected to satisfy
the variational inequality%
\begin{equation}
({\rm OP})\text{ }\left\{ 
\begin{array}{c}
{\min }\left( \rho V(q)-\frac{1}{2}\left( \frac{h-l}{\sigma }\right)
^{2}q^{2}(1-q)^{2}V^{\prime \prime }\left( q\right)
+C(q),\,\,V(q)-G(q)\right) =0,\text{ \ \ }q\in \left[ 0,1\right] , \\ 
\\ 
V\left( 0\right) =G\left( 0\right) \text{ \ \ and \ \ }V\left( 1\right)
=G\left( 1\right);%
\end{array}%
\right.   \label{OP}
\end{equation}%
see, for example \cite{pham2009continuous,reikvam1998viscosity}. For the boundary conditions, we recall that $q=0$ is an absorbing state
and, thus, (\ref{V-G})\ gives

\begin{equation*}
V(0)=\sup_{\tau \in \mathcal{T}^{Y}_0}\mathbb{E}\left. \left[ -\int_{0}^{\tau
}e^{-\rho t}C(0)dt+e^{-\rho \tau }{G(0)}\right\vert q_{0}=0\right] .
\end{equation*}%
Using that $C(q)>0$ and $\rho >0,$ we deduce that the optimal time must be $%
\tau^{\ast} =0$ and $V(0) =G(0)$ follows. Similar
arguments yield that $V(1)=G(1)$.

For the reader's convenience, we highlight below the key steps for the
derivation of (\ref{OP}). If $q\in \left( 0,1\right),$ there are two
admissible, in general suboptimal, policies: i) the DM may immediately choose
product $A$ or $B,$ without seeking any information about the latter or ii)\
she may dedicate some time, say $\left( 0,\varepsilon \right] $ with $%
\varepsilon $ small, learning about the unknown product $B$ before deciding which product to
choose. Choices (i) and (ii) give, respectively, 
\begin{equation}
V(q)\geq G(q) \text{ \ \ and  \ \ } 
V(q)\geq \mathbb{E}\left[ \left. -\int_{0}^{\varepsilon }e^{-\rho
t}C(q_{t})dt+e^{-\rho \varepsilon }V(q_{\varepsilon })\right\vert q_{0}=q%
\right].\label{Inequality1}
\end{equation}%
Assuming that $V$ is smooth enough, It\^{o}'s formula gives, for $t \in (0,\varepsilon)$, 
\begin{equation*}
V(q_{\varepsilon })=V(q)+\int_{0}^{\varepsilon }\frac{1}{2}\left( \frac{h-l}{%
\sigma }\right) ^{2}q_{t}^{2}(1-q_{t})^{2}V^{\prime \prime }\left(
q_{t}\right) dt+\int_{0}^{\varepsilon }\sigma V^{\prime }\left( q_{t}\right)
dZ_{t},
\end{equation*}%
where we used (\ref{belief-dynamics}). Dividing by $\varepsilon $ and passing
to the limit $\varepsilon \downarrow 0,$ we deduce that 
\begin{equation}
\rho V(q)-\frac{1}{2}\left( \frac{h-l}{\sigma }\right)
^{2}q^{2}(1-q)^{2}V^{\prime \prime }\left( q\right) +C(q)\geq 0.
\label{Inequality2}
\end{equation}%
Because one of these two choices must be optimal, (\ref{Inequality1}) or (%
\ref{Inequality2}) must hold as equality and \eqref{OP} follows.

\begin{lemma}\label{lemma:continuity}
    There exists a constant $\rho_0>0$ such that for each $\rho \ge \rho_0$, the value function $V$ is Lipschitz continuous on $%
\left[ 0,1\right] $.
\end{lemma}

\begin{proof}
    We show that there exists a positive constant $\rho
_{0}=\rho _{0}(${$h,l,\sigma )$} such that, for $\rho >2\rho _{0}$, 
\begin{equation*}
|V(q_{1})-V(q_{2})|\leq {\left( \frac{L}{\rho _{0}}+K\right) }|q_{1}-q_{2}|,%
\text{ \ }q_{1},q_{2}\in \lbrack 0,1].
\end{equation*}%
First note that, for $q_{1},q_{2}\in \left[ 0,1\right] $, the function $%
b(q):=\frac{h-l}{\sigma }q(1-q)$ satisfies $b(q)\leq \frac{h-l}{\sigma }q$
and, furthermore,
\begin{equation*}
|b(q_{1})-b(q_{2})|\leq \frac{h-l}{\sigma }((1+q_{1}+q_{2})|q_{1}-q_{2}|)%
\leq 3\frac{h-l}{\sigma }|q_{1}-q_{2}|.
\end{equation*}%
Therefore, (see, for example, \cite[Theorem 1.3.16]{pham2009continuous}),
there exists $\rho _{0}=\rho _{0}(${$h,l,\sigma )>0$} such
that 
\begin{equation*}
\mathbb{E}\left[ \sup_{0\leq u\leq t}|q_{u}^{q_{1}}-q_{u}^{q_{2}}|\right]
\leq e^{\rho _{0}t}|q_{1}-q_{2}|,
\end{equation*}%
{where $q_u^{q_i}$ solves \eqref{belief-dynamics} with initial state $q_i$, $i=1,2$.}
For {$\rho >2\rho _{0}$}, we can, similarly, prove that $\mathbb{E}\left[
\sup_{t\geq 0}e^{-\rho t}|q_{t}^{q_{1}}-q_{t}^{q_{2}}|\right] \leq
|q_{1}-q_{2}|$. In turn, the Lipschitz properties of $C(.)$ and $G(.)$ (see 
Assumption \ref{ass:CI}) yield 
\begin{eqnarray*}
|V(q_{1})-V(q_{2})| &\leq &{\ \sup_{\tau \in \mathcal{T}^{Y}_0}}\,\,\mathbb{E}%
\left[ \int_{0}^{\tau }e^{-\rho t}\left\vert
C(q_{t}^{q_{1}})-C(q_{t}^{q_{2}})\right\vert dt+e^{-\rho \tau }\left\vert
G(q_{\tau }^{q_{1}})-G(q_{\tau }^{q_{2}})\right\vert \right]  \\
&\leq &{\ L}\mathbb{E}\left[ \int_{0}^{\infty }e^{-\rho t}\left\vert
q_{t}^{q_{1}}-q_{t}^{q_{2}}\right\vert dt\right] +{\ K}\mathbb{E}\left[
\sup_{t\geq 0}e^{-\rho t}\left\vert q_{t}^{q_{1}}-q_{t}^{q_{2}}\right\vert %
\right]  \\
&\leq &{\ \left( \frac{L}{\rho -\rho _{0}}+K\right) }|q_{1}-q_{2}|.
\end{eqnarray*}
\end{proof}

The connection between the value function and viscosity solutions of optimal
stopping problems was established in \cite{reikvam1998viscosity} (see, also, 
{ \cite{pham2009continuous}}). Throughout, we will be using viscosity arguments to carry
out an extensive analysis of the problem; for completeness, we also
highlight the key steps in the following characterization result.

\begin{theorem}
    The value function $V:[0,1]\to\mathbb R_+$ (cf. \ref{V-G}) is a
viscosity solution to (\ref{OP}), unique in the class of Lipschitz
continuous functions.
\end{theorem}

\begin{proof}
    We first establish that $V$ is a viscosity subsolution of (%
\ref{OP}) in {$(0,1).$} For this, let $\bar{q}\in \left( 0,1\right) $ and
consider a test function $\varphi \in \mathcal{C}^{2}((0,1))$ such that $%
(V-\varphi )(\bar{q})=\max_{q\in \left( 0,1\right) }(V-\varphi )$ and $%
(V-\varphi )(\bar{q})=0.$ We need to show that 
\begin{equation*}
\min \left( \rho \varphi \left( \bar{q}\right) -\frac{1}{2}\left( \frac{h-l}{%
\sigma }\right) ^{2}\bar{q}^{2}(1-\bar{q})^{2}\varphi ^{\prime \prime
}\left( \bar{q}\right) +C(\bar{q}),V(\bar{q})-G(\bar{q})\right) \leq 0.
\end{equation*}%
We argue by contradiction, assuming that both inequalities, 
\begin{equation}
\rho \varphi \left( \bar{q}\right) -\frac{1}{2}\left( \frac{h-l}{\sigma }%
\right) ^{2}\bar{q}^{2}(1-\bar{q})^{2}\varphi ^{\prime \prime }\left( \bar{q}%
\right) +C(\bar{q})>0\text{ \ and \ }V(\bar{q})-G(\bar{q})>0,
\label{Contradiction}
\end{equation}%
hold. Using the continuity of the involved functions, we deduce that there would then
exist $\delta >0$ such that, for  $0\leq t\leq {\ \hat{\tau}},$ a.s., 
\begin{equation}
\rho \varphi \left( \bar{q}_{t}\right) -\frac{1}{2}\left( \frac{h-l}{\sigma }%
\right) ^{2}\bar{q}_{t}^{2}(1-\bar{q}_{t})^{2}\varphi ^{\prime \prime
}\left( \bar{q}_{t}\right) +C(\bar{q}_{t})\geq \delta \text{ and }V\left( 
\bar{q}_{t}\right) -G(\bar{q}_{t})\geq \delta ,  \label{delta-inequalities}
\end{equation}%
where {$\bar{q}_{t}:=q_{t}^{\bar{q}}{\ }$ solves \eqref{belief-dynamics} with initial state $\bar{q}$}, and ${\hat{\tau}}$ is the exit
time of $\bar{q}_{t}$ from {$[\bar{q}-\delta ,\bar{q}+\delta ]$}. 
Applying It\^{o}'s formula to $e^{-\rho t}\varphi (\bar{q}_{t})$, $t\in %
\left[ 0,{\hat{\tau}}\wedge \tau \right] ,$ $\tau \in \mathcal{T}^{Y}_0,$
yields 
\begin{eqnarray}
V(\bar{q}) &=&\varphi (\bar{q})=\mathbb{E}\left[ \int_{0}^{{\hat{\tau}}%
\wedge \tau }e^{-\rho t}\left(\rho \varphi \left( \bar{q}_{t}\right) -\frac{1}{2}%
\left( \frac{h-l}{\sigma }\right) ^{2}\bar{q}_{t}^{2}(1-\bar{q}%
_{t})^{2}\varphi ^{\prime \prime }\left( \bar{q}_{t}\right) \right)dt+e^{-\rho ({\hat{\tau}}\wedge \tau )}\varphi (\bar{q}_{\tau \wedge {\hat{\tau}}})%
\right]   \notag \\
&\geq &\mathbb{E}\left[ -\int_{0}^{{\hat{\tau}}\wedge \tau }e^{-\rho t}C(%
\bar{q}_{t})dt+e^{-\rho \tau }G(\bar{q}_{\tau })\mathbf{1}_{\tau <{\ \hat{\tau}}}+e^{-\rho {\ \hat{\tau}}}V(\bar{q}_{{\ \hat{\tau}}})\mathbf{1}_{{\ 
\hat{\tau}}\leq \tau }\right]  \notag \\
&& +\delta \mathbb{E}\left[ \int_{0}^{{\hat{\tau%
}}\wedge \tau }e^{-\rho t}dt+e^{-\rho \tau }\mathbf{1}_{\tau <{\hat{\tau}}}%
\right],   \label{eq:V_ine}
\end{eqnarray}%
where we use that $\varphi \geq V$ on {$[\bar{q}-\delta ,\bar{q}+\delta ]$ }%
and (\ref{delta-inequalities}). Next, we claim that there exists $c_{0}>0$
such that 
\begin{equation*}
\mathbb{E}\left[ \int_{0}^{{\hat{\tau}}\wedge \tau }e^{-\rho t}dt+e^{-\rho
\tau }\mathbf{1}_{\tau <{\hat{\tau}}}\right] \geq c_{0},\text{ }\tau \in 
\mathcal{T}^{Y}_0.
\end{equation*}%
{To this end, let $v\left( q\right) :=c_{0}\left( 1-\frac{1}{\delta ^{2}}%
(q-\bar{q})^{2}\right) $ with $c_{0}=\min \left( \left(\rho +\frac{1}{16}\left(\frac{%
h-l}{\sigma\delta }\right)^2\right)^{-1},1\right) $. Then,}%
\begin{equation}
\rho v(q)-\frac{1}{2}\left( \frac{h-l}{\sigma }\right)
^{2}q^{2}(1-q)^{2}v^{\prime \prime }(q)\leq 1,\text{ \ \ }v(q)\leq c_{0}\leq
1,\text{ \ }q\in (\bar{q}-\delta ,\bar{q}+\delta ),  \label{eqn-1}
\end{equation}%
and 
\begin{equation}
v(\bar{q}-\delta )=0\text{, \ }v(\bar{q}+\delta )=0\text{ \ and \ }v(\bar{q}%
)=c_{0}>0.  \label{eq-2}
\end{equation}%
{A}pplying It\^{o}'s formula to $e^{-\rho t}w(\bar{q}_{t})$ gives 
\begin{eqnarray}
0<c_{0}=v(\bar{q}) &=&\mathbb{E}\left[ \int_{0}^{{\ \hat{\tau}}\wedge \tau
}e^{-\rho t}\left( \rho v\left( \bar{q}_{t}\right) -\frac{1}{2}\left( \frac{%
h-l}{\sigma }\right) ^{2}\bar{q}_{t}^{2}(1-\bar{q}_{t})^{2}v^{\prime \prime
}(\bar{q}_{t})\right) dt+e^{-\rho ({\hat{\tau}}\wedge \tau )}v(\bar{q}_{{\ 
\hat{\tau}}\wedge \tau })\right]   \notag \\
&\leq &\mathbb{E}\left[ \int_{0}^{{\ \hat{\tau}}\wedge \tau }e^{-\rho
t}dt+e^{-\rho \tau }\mathbf{1}_{\tau <{\ \hat{\tau}}}\right] ,\text{ }\tau
\in \mathcal{T}^{Y}_0,
\end{eqnarray}%
where the last inequality holds from (\ref{eqn-1}) and (\ref{eq-2}).
Plugging the last inequality into \eqref{eq:V_ine}, taking the supremum
over $\tau \in \mathcal{T}^{Y}_0$ and using \eqref{Contradiction} leads to a contradiction, 
since the Dynamic Programming Principle (DPP) yields
\begin{equation}
V(p)=\sup_{\tau \in \mathcal{T}^{Y}_0}\mathbb{E}\left[ -\int_{0}^{\tau \wedge
\hat \tau}e^{-\rho t}C({q_{t}^{p}})dt+e^{-\rho \tau }G(q_{\tau }^{p})1_{\tau
<\hat \tau}+e^{-\rho \hat \tau}V(q_{\hat \tau}^{p})1_{\hat \tau\leq \tau }\right].
\label{eqn:dpp}
\end{equation}%
The supersolution
property follows easily and the boundary conditions were shown earlier.
Uniqueness follows from the comparison principle; we refer the reader to \cite[Theorem~3.3]{crandall1992user}.
\end{proof}

{\subsection{The ``exploration'' and the ``product - selection'' regions} 

We introduce the sets 
\begin{equation}
\mathcal{S}:= \Big\{ q\in \lbrack 0,1]\,\,\Big\vert\,
\,\,V(q)=G(q)\Big\} \text{ \ and \ }\mathcal{E}:=\Big\{  q\in
\lbrack 0,1]\,\,\Big\vert \,\,\,V(q)>G(q)\Big\}.  \label{S-E-sets}
\end{equation}%
We will refer to $\mathcal{S}$ as the \textit{product-selection}, or \textit{%
stopping} \textit{region}, since it is therein optimal to immediately stop
and choose one of the products. Its complement $\mathcal{E}$ is the \textit{%
continuation or exploration region}, as it is optimal in this region to keep exploring
and acquiring information about the unknown product. In this section, we provide various results for the regions $\mathcal{S}$ and $\mathcal{E}$ for general pairs $(C(.),G(.))$.

\begin{lemma}
\label{lemma:never-stop}
    Let
\[
\widehat V(q) := \E\!\left[- \int_{0}^{\infty} e^{-\rho t} C(q_t)\, dt \right], q\in [0,1].
\]
Then, if $\mathcal{S} = \varnothing $, we have $\widehat V \ge G $ and, in turn, $V = \widehat V$.
\end{lemma}

\begin{proof}
    If $\mathcal{S} = \varnothing$, then $V > G$ everywhere and, therefore, $V$ satisfies
    $$\rho V(q)-\frac{1}{2}\left( \frac{h-l}{\sigma }\right)
^{2}q^{2}(1-q)^{2}V^{\prime \prime }\left( q\right)
+C(q) = 0, \quad \text{on } [0,1].$$
On the other hand, $\widehat V \in \mathcal{C}([0,1])$ satisfies the same linear equation on $[0,1]$ with the same boundary condition. Since $V$ is continuous on $[0,1]$ (by Lemma~\ref{lemma:continuity}), uniqueness of viscosity solutions in the class $\mathcal{C}([0,1])$ \cite{fleming2006controlled} implies that $V = \widehat V$. In particular, $\widehat V = V \ge G$.

If $\widehat V \ge G$, then $\widehat V$ satisfies the same variational inequality as $V$, i.e.
\[
\min\big(\,\rho \widehat V(q)-\frac{1}{2}\left( \frac{h-l}{\sigma }\right)
^{2}q^{2}(1-q)^{2}\widehat V^{\prime \prime }\left( q\right)
+C(q),\; \widehat V(q) - G(q)\,\big) = 0,
\]
and using, once more, the uniqueness of viscosity solutions in the class $\mathcal{C}([0,1])$, we conclude that $V = \widehat V$.
\end{proof}

Next, we study the structure of the stopping region.

\begin{proposition}
    Let $\underline{C} := \inf_{q \in [0,1]} C(q)$.  
If
\[
\quad \sup_{q \in [0,1]} G(q) > -\,\frac{1}{\rho}\underline{C},
\]
then, the stopping region $\mathcal{S}$ is non-empty $(\mathcal{S} \neq \varnothing)$.
\end{proposition}

\begin{proof}
    Since $\sup_{q \in [0,1]} G(q) > -\,\frac{1}{\rho}\underline{C}$, there exists $q^{\ast}$ such that $G(q^{\ast}) > -\,\frac{1}{\rho}\underline{C}$. Then,
\[
\widehat V(q^{\ast})
= \E \!\left[ - \int_{0}^{\infty} e^{-\rho t} C(q^{\ast}_t)\, dt \right]
\le -\int_{0}^{\infty} e^{-\rho t} \underline{C}\, dt
= -\,\frac{\underline{C}}{\rho} \le G(q^{\ast}),
\]
and using Lemma \ref{lemma:never-stop} we conclude.
\end{proof}

\begin{proposition}
    \label{prop:C1_solution} 
    
    Assume that the continuation region $\mathcal E$ is non-empty $(\mathcal{E} \neq \varnothing)$, and that $G$ is $\mathcal C^1 (\mathcal S)$. Then, the value function $V$ is $\mathcal{C}^{1}\left( \partial \mathcal{E}\right) $ and the unique $%
\mathcal{C}^{2}\left( {\mathring{\mathcal{E}}}\right) $ solution to 
\begin{equation}
\rho V(q) - \frac{1}{2}\left( \frac{h-l}{\sigma }\right)
^{2}q^{2}(1-q)^{2}V^{\prime \prime } (q)+C\left( q\right) =0,\text{ \ }q\in
\mathring{\mathcal{E}}\text{,}  \label{V-smooth-eqn}
\end{equation}
{where $\mathring{\mathcal{E}}$ denotes the interior of $\mathcal{E}$}.
\end{proposition}

\begin{proof}
Classical results (see, for example, \cite{fleming2012deterministic}, \cite[Theorem 2.6]{lian2020pointwise} or \cite[Lemma
5]{tang2021exploratory}) yield the existence and uniqueness of a $%
\mathcal{C}^{2}$ solution of (\ref{V-smooth-eqn}), say $w,$ in any open set $%
\mathcal{O}\subset \mathcal{E}$,\ with boundary condition $w=V.$
 On the other hand, the value function $V$ is a continuous viscosity solution
of \eqref{V-smooth-eqn} in $\mathcal O$ with the same boundary condition. By uniqueness, the Dirichlet problem
admits at most one continuous viscosity solution,
therefore, it must be that $w\equiv V$.

To show that $V\in $ $\mathcal{C}^{1}(\partial \mathcal{E)}$ we argue by
contradiction. To this end, for $\hat{q}\in \partial \mathcal{E}$, we have $%
V(\hat{q})=G(\hat{q})$ and $V(p)\geq G(p)$, $p\in \lbrack 0,1]$. Furthermore,%
\begin{equation*}
\frac{V(p)-V(\hat{q})}{p-\hat{q}}\leq \frac{G(p)-G(\hat{q})}{p-\hat{q}},%
\text{ }p<\hat{q}\text{ \ \ and \ }\frac{V(p)-V(\hat{q})}{p-\hat{q}}\geq 
\frac{G(p)-G(\hat{q})}{p-\hat{q}},\text{ }p>\hat{q}.
\end{equation*}%
Therefore, $V_{-}^{\prime }(\hat{q})\leq G^{\prime }(\hat{q})\leq
V_{+}^{\prime }(\hat{q})$. If $V$ is not differentiable at $\hat{q}$, there
must exist $d\in (V_{-}^{\prime }(\hat{q}),V_{+}^{\prime }(\hat{q}))$.
Let, 
\begin{equation*}
\varphi _{\varepsilon }(p):=V(\hat{q})+d(p-\hat{q})+\frac{1}{2\varepsilon }%
(p-\hat{q})^{2}.
\end{equation*}%
Then, $V$ dominates $\varphi _{\varepsilon }$ locally in a neighborhood of $%
\hat{q}$, i.e., $\hat{q}$ is a local minimum of $V-\varphi _{\varepsilon }$.
From the viscosity supersolution property of $V$, we deduce that 
\begin{equation*}
\rho V(\hat{q})-\frac{1}{2\varepsilon }\frac{(h-l)^{2}}{2\sigma ^{2}}\hat{q}%
^{2}(1-\hat{q})^{2}V''(\hat q)+C(\hat{q})\geq 0.
\end{equation*}%
Sending $\varepsilon \rightarrow 0$ yields a contradiction since $\hat{q}\in
(0,1)$, and both $V$ and $C$ are Lipschitz continuous on $\left[ 0,1%
\right] .$ If $\mathcal E=\varnothing$, the result holds trivially with $V\equiv G$.
\end{proof}

}

\section{The single (irreversible) decision problem with general information acquisition costs}
\label{sec:single-decision}
Next, we revisit the only case that has been so far studied but we extend it to arbitrary, general information costs. Specifically, we study \eqref{V-IR-generalG}  with \textit{general} function $C(.)$, which is equivalent to 
 \eqref{V-G} with $G$ given by 
\begin{equation}
G(q)={\max \Big(\mu ,qh+(1-q)l
\Big).}  \label{G-linear}
\end{equation}

We explore the structure of regions $\mathcal{S}$ and $\mathcal{E}$ (cf. \eqref{S-E-sets}), and investigate their dependence on the model parameters.
To our knowledge, only the constant case $C(q) = C_I, \ q \in [0,1]$ has been so far analyzed in the literature (see \cite{moscarini2001optimal,keppo2008demand}). 

The analysis is carried out using viscosity appropriate sub- and super-
solutions. To this end, 
we introduce point $\hat p \in (0,1)$, 
\begin{equation}
\label{eq:p}
\hat p:=\frac{\mu -l}{h-l},
\end{equation}
which we will use frequently later on. Note that at this point, $G(\hat p ) = \hat p h + (1- \hat p)l= \mu$.

\begin{proposition}
    \label{lem:supersolution} 
    The regions $\mathcal{S}$ (stopping) and $\mathcal{E}$ (exploration) are non-empty. 
    Moreover, there exists $\varepsilon>0$ such that
    $[0,\varepsilon] \cup [1-\varepsilon,1] \subseteq \mathcal{S}$ and $(\hat{p}-\varepsilon,\ \hat{p}+\varepsilon) \subseteq \mathcal{E}$, with $\hat p$ as in \eqref{eq:p}.
\end{proposition}

\begin{proof}
    i) {$[0,\varepsilon] \cup [1-\varepsilon,1] \subseteq \mathcal{S}$}.

\noindent We show that there exists a continuous viscosity supersolution $v${, such
that $v=G$ on $[0,\varepsilon ]\cup \lbrack 1-\varepsilon ,1],$ for
sufficiently small $\varepsilon >0$. To this end, }for some $M>0$ and $%
\varepsilon >0$ to be determined, let%
\begin{equation*}
v(q):=
\begin{cases}
    \mu +M\left( \frac{q}{\varepsilon }-1\right) _{+}^{3},\text{ \ }&0\leq
q\leq 2\varepsilon ,\\
\mu +M+\frac{q-2\varepsilon }{1-4\varepsilon }(h(1-2\varepsilon
)+2l\varepsilon -\mu ),\text{ \ }&2\varepsilon \leq q\leq 1-2\varepsilon ,\\
q h +(1-q)l+M\left( \frac{1-q}{\varepsilon }-1\right) _{+}^{3},\text{ \ }%
&1-2\varepsilon \le q \le 1.
\end{cases}
\end{equation*}%
The structure of $v$ is shown in Figure \ref{fig:v_region_nonempty}. 
\begin{figure}[H]
    \centering
    \includegraphics[width=0.45\linewidth]{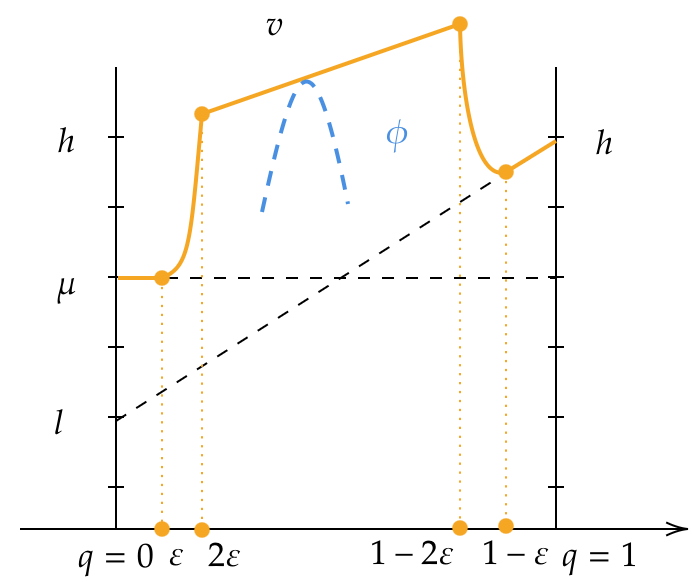}
    \caption{Viscosity supersolution $v$.}
    \label{fig:v_region_nonempty}
\end{figure}

Note that, by construction, $v$ is continuous and twice differentiable at $%
q=\varepsilon $ and $q=1-\varepsilon $. Furthermore,%
\begin{equation*}
v^{\prime }(q)=3\frac{M}{\varepsilon }\left( \frac{q}{\varepsilon }-1\right)
_{+}^{2}\text{ \ and \ }v^{\prime \prime }(q)=6\frac{M}{\varepsilon ^{2}}%
\left( \frac{q}{\varepsilon }-1\right) _{+},\text{\ \ }0<q<2\varepsilon ,
\end{equation*}%
\begin{equation*}
v^{\prime }(q):=\frac{1}{1-4\varepsilon }(h(1-2\varepsilon )+2l\varepsilon
-\mu )\text{ \ and }v^{\prime \prime }\left( q\right) =0,\text{ }%
2\varepsilon <q<1-2\varepsilon ,
\end{equation*}%
\begin{equation*}
v^{\prime }(q):=h-l-3\frac{M}{\varepsilon }\left( \frac{1-q}{\varepsilon }%
-1\right) _{+}^{2}\text{ \ and \ } v^{\prime \prime }\left( q\right) = 6\frac{M}{\varepsilon ^{2}}\left( \frac{1-q%
}{\varepsilon }-1\right) _{+},\text{\ }1-2\varepsilon <q<1.
\end{equation*}%
Then, for $q\in (2\varepsilon ,1-2\varepsilon )$, 
\begin{equation*}
\rho v(q)-\frac{1}{2}\left( \frac{h-l}{\sigma }\right)
^{2}q^{2}(1-q)^{2}v^{\prime \prime }\left( q\right) +C \left( q\right)
=\rho v(q)+C (q)>0,
\end{equation*}%
and, for $q\in \lbrack 0,2\varepsilon )\cup (1-2\varepsilon ,1]$, 
\begin{equation*}
\rho v(q)-\frac{1}{2}\left( \frac{h-l}{\sigma }\right)
^{2}q^{2}(1-q)^{2}v^{\prime \prime }\left( q\right) +C \left( q\right)
\geq \rho \mu -6\frac{M}{2\varepsilon ^{2}}\left( \frac{h-l}{\sigma }\right)
^{2}(2\varepsilon )^{2}\geq \rho \mu -12M\left( \frac{h-l}{\sigma }\right)
^{2}.
\end{equation*}
Setting $M:=\frac{\rho \mu }{24}\left(\frac{\sigma}{h-l}\right)^2$, we have that $\rho \mu -12M\left( \frac{h-l}{\sigma }\right)^{2}>0$.

Next, note that $v_{-}^{\prime }(2\varepsilon )=3M\varepsilon ^{-1}$ and $%
v_{+}^{\prime }(1-2\varepsilon )=-3M\varepsilon ^{-1}+h-l$. Therefore, by
choosing $\varepsilon <\min \left( \frac{3M}{2(\mu -l)},\frac{3M}{2(h-\mu )},%
\frac{1}{8}\right) $, it holds that 
\begin{equation*}
v_{-}^{\prime }(2\varepsilon )>\left. \left( \mu +M+\frac{q-2\varepsilon }{%
1-4\varepsilon }(h(1-2\varepsilon )+2l\varepsilon -\mu )\right) _{+}^{\prime
}\right\vert _{q=2\varepsilon }
\end{equation*}%
\begin{equation*}
=\left. \left( \mu +M+\frac{q-2\varepsilon }{1-4\varepsilon }%
(h(1-2\varepsilon )+2l\varepsilon -\mu )\right) _{-}^{\prime }\right\vert
_{q=1-2\varepsilon }
\end{equation*}%
\begin{equation*}
=\frac{h(1-2\varepsilon )+2l\varepsilon -\mu }{1-4\varepsilon }%
>v_{+}^{\prime }(1-2\varepsilon ).
\end{equation*}%
Thus, any $\mathcal{C}^{2}(0,1)$ test function $\varphi $ can only touch $v$ from below at some $q_{0}$ in $[0,2\varepsilon )\cup (2\varepsilon
,1-2\varepsilon )\cup (1-2\varepsilon ,1]$, on which $v$ is $\mathcal{C}^{2}$
and $v^{\prime \prime }(q_{0})\geq \varphi ^{\prime \prime }(q_{0})$.
Therefore, $v$ is a classical supersolution with $v=G$ on $[0,\varepsilon ]\cup
\lbrack 1-\varepsilon ,1]$. As a consequence, we have, by the comparison principle for
viscosity solutions in the class $\mathcal C([0,\varepsilon ]\cup
\lbrack 1-\varepsilon ,1])$, that
the value function satisfies $G\leq V\leq v$ and, hence, $V=G$, $q\in \lbrack
0,\varepsilon ]\cup \lbrack 1-\varepsilon ,1]$. We conclude. \hfill

ii) {$(\hat{p}-\varepsilon,\ \hat{p}+\varepsilon) \subseteq \mathcal{E}$}.

\noindent We show that there exists a continuous viscosity subsolution $u$, such
that $u(\hat p)>G(\hat p)$. To this
end, for some $M>0$ and $\varepsilon >0$ to be chosen, let 
\begin{equation*}
r(q):=\mu +M\left( -\frac{\varepsilon }{2}+\frac{1}{\varepsilon }%
(q-(\hat p-\varepsilon ))^{2}\right) ,\text{ \ }\hat p-\varepsilon \leq q\leq
\hat p+\varepsilon ,
\end{equation*}%
with%
\begin{equation*}
\varepsilon <\min \left( \frac{\hat p}{2},\frac{1}{32(15M+\overline{C}+\rho
\mu )}\left( \frac{h-l}{\sigma }\right) ^{2}\hat p^{2}(1-\hat p)^{2}M\right) \text{ \
\ and \ \ }M<\frac{2}{7}(h-l),
\end{equation*}%
where $\overline{C}:=\max_{|q-\hat p|\leq \varepsilon }C(q)$. 

We claim that $r$ is a viscosity subsolution in $[\hat p-\varepsilon , \hat p+\varepsilon ]$. Indeed, we have $r^{\prime \prime }\left( q\right) =\frac{2M}{%
\varepsilon }$ and 
\begin{align*}
& \rho r(q)-\frac{1}{2}\left( \frac{h-l}{\sigma }\right)
^{2}q^{2}(1-q)^{2}r^{\prime \prime }\left( q\right)  +C(q)\\
& \qquad \leq \rho (\mu +15\varepsilon M)+\overline{C}-\frac{M}{%
2\varepsilon }\left( \frac{h-l}{\sigma }\right) ^{2}(\hat p-\varepsilon
)^{2}(1-\hat p-\varepsilon )^{2}<0,
\end{align*}%
where the last inequality follows from the choice of $\varepsilon $. Next, let 
\begin{equation*}
u(q):=%
\begin{cases}
\max \left( r(q),G(q)\right) , & \hat p-\varepsilon \leq q\leq \hat p+\varepsilon , \\ 
G(q), & \text{otherwise}.%
\end{cases}%
\end{equation*}%
We depict $u$ in the graph below.
\begin{figure}[H]
    \centering
    \includegraphics[width=0.45\linewidth]{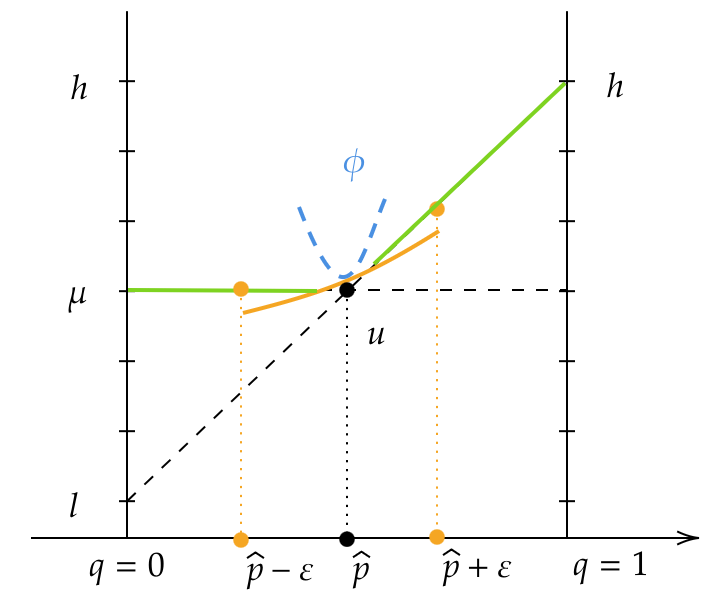}
    \caption{Viscosity subsolution $u$.}
    \label{fig:u_region_nonempty}
\end{figure}
\noindent By construction, $u$ is continuous {on $[0,1]$} since 
\begin{align*}
u(\hat p-\varepsilon )& =\mu -\frac{1}{2}M\varepsilon <\mu =G(\hat p-\varepsilon )\\
\intertext{and}
u(\hat p+\varepsilon )& =\mu +\frac{7}{2}M\varepsilon <\mu +(h-l)\varepsilon
=G(\hat p+\varepsilon ).
\end{align*}%
Therefore, $u$ is a continuous subsolution and, furthermore, $u(\hat p)=r(\hat p)=\mu +%
\frac{1}{2}M\varepsilon >\mu =G(\hat p).$ We easily deduce, by the comparison principle for viscosity solutions in the class of
continuous functions on $[0,1]$ satisfying the same boundary condition, that
$V(\hat p)\geq u(\hat p)>G(\hat p)$, and we conclude. 
\end{proof}

\bigskip

\begin{theorem}
    \label{thm:no-regret} 
    Let the value function $V$ as in \eqref{V-G} and the payoff function $G$ as in \eqref{G-linear}. The following assertions hold:
\begin{enumerate}
    \item There exist points $\underline{q}$ and $\overline{q},$ with $0<%
\underline{q}<\overline{q}<1,$ such that 
\begin{equation}
V(q)=G(q),\text{ }q\leq \underline{q}\text{, }q\geq \overline{q}\text{ \ \
and \ \ }V(q)>G(q),\text{ }q\in \lbrack \underline{q},\overline{q}].
\label{V-q}
\end{equation}
    \item $\ V$ is convex and non-decreasing on $[0,1]$.
\end{enumerate}
\end{theorem}

\begin{proof}
    i)\ Recall from \eqref{OP} that $V(0)=\mu.$ We claim that, if there exists $%
\underline{q}>0$ such that $V(\underline{q})=\mu $, then it must be 
$V(q)=\mu ,\text{ \ }q\in \lbrack 0,\underline{q}]$.
We argue by contradiction, assuming there exists $p$ such that $%
V(p):=\sup_{p\in \lbrack 0,\underline{q}]}V(p)>\mu $. From the DPP (cf. \eqref{eqn:dpp}), we have that, for each $\hat \tau \in \mathcal{T}^{Y}_0$,
\begin{equation*}
V(p)=\sup_{\tau \in \mathcal{T}^{Y}_0}\mathbb{E}\left[ -\int_{0}^{\tau \wedge
\hat \tau}e^{-\rho t}C({q_{t}^{p}})dt+e^{-\rho \tau }G(q_{\tau }^{p})1_{\tau
<\hat \tau}+e^{-\rho \hat \tau}V(q_{\hat \tau}^{p})1_{\hat \tau\leq \tau }\right],
\end{equation*}%
where $q_{t}^{p}$ solves \eqref{belief-dynamics} starting at $p$. Then, for $\hat \tau:=\inf \left( t\,\,\geq 0|\,\,q_{t}^{p}=0\mathrm{\ }\text{or }%
q_{t}^{p}=\underline{q}\right) $, we obtain 
\begin{equation*}
V(p)=\sup_{\tau \in \mathcal{T}^{Y}_0}\mathbb{E}\left[ -\int_{0}^{\tau \wedge
\hat \tau}e^{-\rho t}C(q_{t}^{p})dt+e^{-\rho \tau }G(q_{\tau }^{p})1_{\tau
<\hat \tau}+e^{-\rho \hat \tau}\mu 1_{\hat \tau\leq \tau }\right] ,
\end{equation*}%
where we used that $V(q_{\hat \tau}^{p})1_{\hat \tau\leq \tau }=\mu 1_{\hat \tau\leq \tau }$.
Conditionally on the event $\{\tau <\hat \tau\}$, we have a.s. that $%
0\leq q_{\tau }^{p}\leq \underline{q}\leq \frac{\mu -l}{h-l}$ and, hence, $%
G(q_{\tau }^{p})=\mu $. This, however, yields a contradiction since 
\begin{eqnarray*}
V(p) &=&\sup_{\tau \in \mathcal{T}^{Y}_0}\mathbb{E}\left[ \int_{0}^{\tau \wedge
\hat \tau}-e^{-\rho t}C(q_{t}^{p})dt+e^{-\rho \tau }\mu 1_{\tau <\hat \tau}+e^{-\rho \hat \tau}\mu
1_{\hat \tau\leq \tau }\right]  \\
&=&\sup_{\tau \in \mathcal{T}^{Y}_0}\mathbb{E}\left[ \int_{0}^{\tau \wedge
\hat \tau}-e^{-\rho t}C(q_{t}^{p})dt+e^{-\rho \tau \wedge \hat \tau}\mu \right] \leq \mu ,
\end{eqnarray*}%
where the last inequality follows from $C(q)>C_1>0$. 

Similarly, we can show that
if there exists $\overline{q}\in (0,1)$ such that the $V(\overline{q})=%
\overline{q}h+(1-\overline{q})l$, then we must have 
$V(q)=qh+(1-q)l,\text{ \ }q\in \lbrack \overline{q},1]$.
Using the above results we deduce that the continuation region must satisfy $%
\mathcal{E}=${$(\underline{q},\overline{q}).$}

From Proposition \ref{prop:C1_solution}, we have that $V$ satisfies 
\begin{equation*}
\rho V (q)-\frac{1}{2}\left( \frac{h-l}{\sigma }\right)
^{2}q^{2}(1-q)^{2}V^{\prime \prime }+C(q)=0\text{ \ \ in \ }\mathcal{E}.
\end{equation*}

Since $V\left( q\right) \geq G(q)>0$ and $C(q)>0,$ $q\in \lbrack 0,1]$, the
above gives that $V^{\prime \prime }(q)>0,$ $q\in \mathcal{E}$, and thus $V$
is strictly convex in $\mathcal{E}$. Furthermore, $V$ is constant on $[0,\underline{q}%
]$ and linear on $[\overline{q},1]$. Therefore, to establish the convexity
on $\left[ 0,1\right] ,$ it suffices to show that $V$ is convex at both $%
\underline{q}$ and $\overline{q}$. To this end, for any $(a,b,\lambda )$
such that $\underline{q}=\lambda a +(1-\lambda ) b,$ with $0<a<\underline{q}<b<1$
and $\lambda \in (0,1)$, we have 
\begin{equation*}
\lambda V(a)+(1-\lambda )V(b)\geq \mu =V(\underline{q}),
\end{equation*}%
since $V(b)\geq \mu $. Similarly, for any $(a,b,\lambda )$ such that $%
\overline{q}=\lambda a +(1-\lambda ) b,$ with $0<a<\overline{q}<b<1$ and $\lambda
\in (0,1)$, we deduce that
\begin{equation*}
\lambda V(a)+(1-\lambda )V(b)=\lambda V(a)+(1-\lambda )\left( V(\overline{q}%
)+(h-l)(b-\overline{q})\right) 
\end{equation*}%
\begin{equation*}
=\lambda V(a)+(1-\lambda )\left( V(\overline{q})+\lambda (h-l)(b-a)\right)
>V\left( \bar{q}\right) ,
\end{equation*}%
where the last inequality holds since $V(a)>ha+l(1-a).$

The monotonicity follows easily as $V^{\prime }(q)\geq 0$ on $\mathcal{E}$
and $V$ is non-decreasing on $\mathcal{S}$. 
\end{proof}

 \begin{corollary}
     Let points $\underline{q}$ and $\overline{q}$ be as in Theorem \ref{thm:no-regret}. Then, the value function $V$ in \eqref{V-G} is the unique $%
\mathcal{C}^{2}\left( \left( \underline{q},\overline{q}\right) \right) $
solution of%
\begin{equation}
\begin{cases}
& \rho V (q)-\frac{1}{2}\left( \frac{h-l}{\sigma }\right)
^{2}q^{2}(1-q)^{2}V^{\prime \prime }(q)+C(q)=0, \\ 
& V(\underline{q})=\mu ,V(\overline{q})=\overline{q} h + (1-\overline{q} )l, \\ 
& V^{\prime }(\underline{q})=0,V^{\prime }(\overline{q})=h-l,%
\end{cases}
\label{Stephan}
\end{equation}%
in the class of Lipschitz continuous functions.
\end{corollary}

To facilitate the presentation, we introduce the notation 
\begin{equation}
\mathcal{S}_{1}:=[0,\underline{q}]\text{ \ \ and \ }\mathcal{S}_{2}:=[%
\overline{q},1].  \label{S1-S2}
\end{equation}

\textit{Discussion: }{Theorem \ref{thm:no-regret}} implies that
the DM will choose product $A$ if the initial belief $q\in \mathcal{S}_{1}$
and product $B$ if $q\in \mathcal{S}_{2}$. If, on the other hand, $q\in
\left( \underline{q},\overline{q}\right) $ the DM commences learning about the
unknown product $B$ and makes a decision when the belief process hits either one
of the cut-off points, $\underline{q}$\ and $\overline{q}$.

\medskip

We will be calling $[0,\underline{q}]$\ the \textit{safe choice} region, $[%
\underline{q},\overline{q}]$\ the \textit{exploration} region, and $[%
\overline{q},1]$\ the \textit{new choice} region. In general, calculating $%
\underline{q}$\ and $\overline{q}$\ in closed form is not possible, even under simplified assumptions for the information cost $C$ and the payoff function $G.$
One of the main contributions herein is that, despite this lack of
tractability, we are still able to study the behavior of the solution and
the various regions for general information acquisition costs.

\bigskip 

\begin{remark}
    In the degenerate case $\sigma \equiv 0$, the DM can immediately
observe the true value of $\Theta $ as soon as she has access to the
signal process $Y$, which now degenerates to $dY_t=\Theta dt$. If the DM starts
with belief $q_{0_{-}}=q$ at time $t=0_{-}$ and has access to $Y$ at time $t=0
$, then the belief follows the c\`{a}dl\`{a}g process 
\begin{equation}
q_0=%
\begin{cases}
1, & \text{ if }\Theta =h, \\ 
0, & \text{ if }\Theta =l,%
\end{cases}%
\qquad \text{ and }q_t=q_0,\text{ }t\geq 0.  \label{q-disc}
\end{equation}%
Given the possible discontinuity of the belief process at time $t=0$, the
corresponding value function is now defined as 
\begin{equation}
V(q)=\sup_{\tau \in \mathcal{T}^{Y}_0}\mathbb{E}\left. \left[ -\int_{0}^{\tau
}e^{-\rho t}C(q_t)dt+e^{-\rho \tau }{G(q_\tau )}\right\vert q_{0_{-}}=q%
\right] .  \label{V-disc}
\end{equation}%
We argue that, in this case, the optimal stopping time $\tau ^{\ast }=0$.
Indeed, from (\ref{q-disc}) it holds, almost surely, that 
\begin{eqnarray*}
e^{-\rho t}C(q_t) &=&e^{-\rho t}C(q_0)\leq C(q_0),\text{ }%
t\geq 0,  \notag \\
e^{-\rho t}G(q_t) &=&e^{-\rho t}G(q_0)\leq G(q_0),\text{ }t\geq 0.
\end{eqnarray*}%
In turn, (\ref{V-disc})\ yields 
\begin{equation*}
V(q)=\mathbb{E}\left[ \left. G\left( q_0\right) \right\vert q_{0_{-}}=q%
\right] =qG(1)+(1-q)G(0)=qh+(1-q)\mu .
\end{equation*}
\end{remark}

\subsection{Sensitivity analysis for general information costs }
\label{sec:precise_information}

We analyze the effects of $\rho ,\sigma $ and $C(.)$ on the regions $%
\mathcal{S}_{1},\mathcal{S}_{2}$ and $\mathcal{E}$ for \textit{arbitrary}
information cost functions $C(.)$ and payoff functions $G$ of form (\ref%
{G-linear}). 
The solution approach is based on building appropriate viscosity sub- and
super- solutions. For the reader's convenience, we put the proofs
in the Electronic Companion.

The following auxiliary result will be used repeatedly.

\begin{lemma}
    \label{lem:comparison}Let $V^{(1)},V^{(2)}\in \mathcal{C}([0,1])$
satisfying, respectively, 
\begin{equation*}
V^{(i)}(q)=G(q),\ \  q\in \lbrack 0,\underline{q}_{i}]\cup \lbrack 
\overline{q}_{i},1] \text{ \ and }\ V^{(i)}(q)>G(q),\text{ }q\in (\underline{q}%
_{i},\overline{q}_{i}),\text{ }i=1,2.
\end{equation*}%
If $V^{(1)}\left( q\right) \geq V^{(2)}\left( q\right) $, $q\in \left[ 0,1\right]
,$ then it must be that 
\begin{equation*}
\underline{q}_{1}\leq \underline{q}_{2}\ \text{and }\ \overline{q}_{2}\leq \overline{q}%
_{1}.
\end{equation*}
\end{lemma}

\begin{proof}
    For every $q\in (\underline{q}_{2},\overline{q}_{2})
$, it holds that $V^{(1)}(q)\geq V^{(2)}(q)>G(q)$, and, thus, $q\in (\underline{q%
}_{1},\overline{q}_{1})$. Therefore, $(\underline{q}_{2},\overline{q}%
_{2})\subseteq (\underline{q}_{1},\overline{q}_{1})$. 
\end{proof}

The next results yield several monotonicity properties of the exploration region with respect to model inputs  $\rho,\ \sigma, \ C$ and $\mu$.

\begin{proposition}
    \label{prop:constant} The following
assertions hold: 
\begin{enumerate}
    \item If $\rho _{1}\leq \rho _{2},$ then $\underline{q}_{1}\leq $ $\underline{%
q}_{2}$ while $\overline{q}_{1}\geq $ $\overline{q}_{2}$. 
\label{prop:i}
\item If $\sigma _{1}\leq \sigma _{2},$ then $\underline{q}_{1}\leq $ $%
\underline{q}_{2}$ while $\overline{q}_{1}\geq $ $\overline{q}_{2}$. 
\label{prop:iii copy(1)}
\item If $C_{1}(q)\geq $ $C_{2}(q),$ for each $q\in \left[ 0,1\right]
,$ then $\underline{q}_{1}\leq $ $\underline{q}_{2}$ while $\overline{q}_{1}%
\geq $ $\overline{q}_{2}.$ 
\label{prop:ii copy(1)}
\end{enumerate}
\end{proposition}

\begin{proposition}\label%
{prop:constant:mu} 
{If $\mu _{1}\leq \mu _{2},$ then $\underline{q}_{1}\leq $ $%
\underline{q}_{2}$ while $\overline{q}_{1}\leq $ $\overline{q}_{2}$. }
\end{proposition}

\textit{Discussion:\ }When the discount rate $\rho $ increases, the width $%
\underline{q}-\overline{q}$ of the exploration region $\mathcal{E}$
decreases. {Hence, the DM tends to allocate less \textit{effort} in learning and processing information about the product 
$B$.} Here, the effort refers to the DM's willingness to
remain in the exploration region before stopping, as reflected by the \textit{length} of
$\mathcal E$. Analogous behavior occurs for larger information cost functions. 

If the volatility $\sigma $ is too high to extract useful information, the
DM will spend less effort in learning and prefers to make a product selection
faster. In contrast, the DM spends more effort in learning when $%
\sigma $ is low, as the signal contains more precise information about
product $B$.

{ When the value $\mu $\ of product $A$\ is higher
than the expected value of $B$, the DM is more reluctant to leave
the safe choice region $\mathcal{S}_{1}$ \ (since $\underline{q}$\
increases) versus choosing $B$ \ (since $\overline{q}$\
increases).}

{ In addition to the monotonicity properties in Proposition \ref{prop:constant} and Proposition \ref{prop:constant:mu}, we also derive various limiting results. Their proofs are provided in
the Electronic Companion.

\begin{proposition}
    \label{prop:constant:limit} Let $\hat p = \frac{\mu-l}{h-l}$ (cf. \eqref{eq:p}). The following assertions hold:
    \begin{enumerate}
        \item If $\rho \uparrow +\infty $, then $\underline{q}\rightarrow \hat p$ and $%
\overline{q}\rightarrow \hat p$. \label{prop:limit:i}
        \item If $\sigma \uparrow +\infty $, then $\underline{q}\rightarrow \hat p$ and $%
\overline{q}\rightarrow \hat p$.
        \item Let $0 \leq C _1 \le C (q) \le C _2$ (cf. Assumption \ref{ass:CI}). If $C _1 \uparrow +\infty $, then $\underline{q}\rightarrow \hat p$ and $%
\overline{q}\rightarrow \hat p$. \label{prop:limit:ii}
    \end{enumerate}
\end{proposition}

The proof is provided in the Electronic Companion.

Proposition \ref{prop:constant:limit} implies that $\hat p$ is a critical value that $\underline{q}$ and $\overline{q}$ converges to when $\rho,\sigma, C_1 \uparrow +\infty $.

\begin{proposition}
    \label{prop:constant-2} The following assertions hold:
    \begin{enumerate}
        \item If $l\uparrow \mu $, then $\overline{q}\downarrow 0$ and $\underline{q}%
\downarrow 0$. \label{prop-2:i}
        \item If $h\uparrow \infty $, then $\overline{q}\uparrow 1$ and $\underline{q%
}\downarrow 0$.\label{prop-2:ii}
    \end{enumerate}
\end{proposition}

We note that the above results are by no means trivial as the variables $l$
and $h$ appear in both the ``volatility''\ term $\left( \frac{h-l}{%
\sigma }\right) ^{2}$ and the obstacle term in (\ref{OP}).

Monotonicity of the cutoff points with respect to parameters $h$ and $l$
does not, in general, hold. At the moment, we can only show that $\underline{%
q}$ decreases when $h$ increases and that $\overline{q}$ decreases when $l$
increases.}

\subsection{Constant information acquisition cost}
\label{sec:constant} 
We revisit the well studied case of constant information cost and
piecewise linear reward function, rewritten below for convenience,
\begin{equation}
V(q)=\sup_{\tau \in \mathcal{T}^{Y}_0}\mathbb{E}\left[ \left. -\int_{0}^{\tau
}e^{-\rho t}C(q_{t})dt+e^{-\rho \tau }{G}(q_{\tau })\,\right\vert
q_{0}=q\right] ,  \label{V-constant}
\end{equation}
where
\begin{equation}
C(q)=C_{I}>0\text{ \ \ and \ }G(q)=\max \Big( \mu ,qh+(1-q)l\Big) .
\label{special form}
\end{equation}%
This case was analyzed in {\cite{keppo2008demand}} but we provide herein further
analytical results. 
To this end, under (\ref{special form}), the variational inequality (\ref{OP}) becomes 
\begin{equation*}
{\min }\left( \rho V(q)-\frac{1}{2}\left( \frac{h-l}{\sigma }\right)
^{2}q^{2}(1-q)^{2}V^{\prime \prime }(q)+C_{I},V(q)-\max \Big( \mu
,qh+(1-q)l \Big)\right) =0.
\end{equation*}%
It is easy to deduce that the ODE
\begin{equation*}
\rho V(q)-\frac{1}{2}\left( \frac{h-l}{\sigma }\right)
^{2}q^{2}(1-q)^{2}V^{\prime \prime }(q)+C_{I}=0,
\end{equation*}%
has a general solution of the form%
\begin{equation*}
V(q)=c_{1}v_{1}(q)+c_{2}v_{2}(q)-\frac{C_{I}}{\rho },
\end{equation*}%
where 
\begin{equation*}
v_{1}(q)=\frac{(1-q)^{\frac{k+1}{2}}}{q^{\frac{k-1}{2}}}\text{ \ \ and \ \ }%
v_{2}(q)=\frac{q^{\frac{k+1}{2}}}{(1-q)^{\frac{k-1}{2}}}
\end{equation*}
and
\begin{equation}
\label{eqn:k}
    \text{ \ }k=\sqrt{%
1+8\rho \left(\frac{\sigma}{h-l}\right)^{2}}.
\end{equation}%
The constants $c_{1},c_{2}$ are two free parameters determined from (\ref%
{Stephan}). Thus, we must have the conditions 
\begin{equation*}
\begin{cases}
-\frac{C_{I}}{\rho }+c_{1}v_{1}(\underline{q})+c_{2}v_{2}(\underline{q}) & 
=\mu , \\ 
c_{1}v_{1}^{\prime }(\underline{q})+c_{2}v_{2}^{\prime }(\underline{q}) & =0,
\\ 
-\frac{C_{I}}{\rho }+c_{1}v_{1}(\overline{q})+c_{2}v_{2}(\overline{q}) & =%
\overline{q}h+(1-\overline{q})l, \\ 
c_{1}v_{1}^{\prime }(\overline{q})+c_{2}v_{2}^{\prime }(\overline{q}) & =h-l.%
\end{cases}%
\end{equation*}
After specifying the special form \eqref{special form}, semi-explicit solutions may be found which can be, in turn, used to perform sensitivity analysis and, also, investigate the limiting behavior of the free boundary and the value function. Because the calculations are rather tedious, we choose not to carry out this analysis here but, rather, analyze similar systems in the next section where we study the general problem.

\section{The general model: costly sequential information acquisition and (ir)reversibility of decisions}\label{sec:reversible_decisions}

We provide an extensive analysis of the new model introduced in Section \ref{sec:nested}, which
incorporates both \textit{sequential decisions} and intertemporal \textit{%
differential information} sources. For tractability, we make the following simplifying model assumption if
product $B$ is chosen:
\begin{eqnarray}\label{eq:simplicifation}
    \widetilde{R}\equiv R>0, \quad \widetilde{C}\equiv0,  \quad \widetilde{m}(q) = q h + (1-q)l.
\end{eqnarray}
Variations and relaxations of this setting are discussed in Section~\ref{sec:conclusion}, which are left for future research.

With the simplification in \eqref{eq:simplicifation}, following Definition~\ref{def:nested}, the value function of the the framework follows
\begin{equation}
V(q)=\sup_{\tau \in \mathcal{T}^{Y}_0}\mathbb{E}\left[ \left. -\int_{0}^{\tau
}e^{-\rho t}C(q_{t})dt+e^{-\rho \tau }V_1 \left( q_{\tau
}\right) \Big) \right\vert q_{0}=q\right] ,\text{ \ }q\in \left[ 0,1\right]
.  \label{V-general}
\end{equation}%
where $V_1(q)=\max(V_1^A(q),V_1^B(q)),$  $V_1^A(q)= \mu$, and $V_1^B(q) = U^{B}({q})$ with the nested value
function defined by
\begin{equation}
U^{B}(\widetilde{q})={\sup_{\widetilde{\tau }\in \mathcal{T}^{\widetilde{Y}}_0}}\,%
\mathbb{\widetilde{E}}\left[ \left. \int_{0}^{\widetilde{\tau }}\rho e^{-\rho
t}\left( \widetilde{q}_{t}h+(1-\widetilde{q}_{t})l\right) \,dt+e^{-\rho 
\widetilde{\tau }}\left( \mu -R\right) \,\,\right\vert \,\,\widetilde{q}_0=%
\widetilde{q}\right] ,\text{ \ \ }\widetilde{q}\in \left[ 0,1\right].
\label{V-B-dfn}
\end{equation}%
The running term $\rho e^{-\rho
t}\left( \widetilde{q}_{t}h+(1-\widetilde{q}_{t})l\right) $ models the
discounted rate (per unit of time) of utility from using product $B,$ free
of any information acquisition costs. The term $e^{-\rho \widetilde{\tau }%
}\left( \mu -R\right) $ represents the discounted payoff received at
stopping time $\widetilde{\tau},$ which consists of the known payoff $\mu $
minus the exchange fee $R.$ To make the problem non-trivial, we assume
that 
$$0<l< \mu -R.$$

We study two types of informational signals in the second
exploratory regime associated with $U^B$: a {\it Poisson} signal, which provides the updated information
through its jumps (see Section \ref{sec:poisson}), and a {\it Gaussian} signal, which provides similar information
as the one used in Section \ref{sec:framework} but with a smaller variance (see Section \ref{sec:higher accuracy}).

\subsection{Case 1: Sequential information signal of Poisson type}\label{sec:poisson}

Information acquisition problems with Poisson-type signals have been
proposed, among others, in \cite{horner2017learning,keller2010strategic}. Herein, we consider the case that the belief process $ \widetilde{q}$ follows the dynamics
\begin{equation}\label{eq:belief-onlyPoisson}
d\widetilde{q}_{t}=(1-\widetilde{q}_{t})\,d\,J_{t}^{1}(\lambda \widetilde{q}%
_{t})+(-\widetilde{q}_{t})\,d\,J_{t}^{0}(\lambda (1-\widetilde{q}_{t})),\ \ 
\widetilde{q}_{0}=\widetilde{q}\in \left[ 0,1\right] ,
\end{equation}%
where $J^{1}(\lambda \widetilde{q}_{t})$ and $J^{0}(\lambda (1-\widetilde{q}%
_{t}))$ are independent Poisson counting processes with respective intensity
rates $\lambda \widetilde{q}_{t}$ and $\lambda (1-\widetilde{q}_{t}),$ $%
\lambda >0,$ and $J_{0}^{0}=J_{0}^{1}=0$. We may think about these rates as
\begin{equation*}
\lambda (1-\widetilde{q}_t)dt=\mathbb{\widetilde{P}}\left[ \left.
J_{t+dt}^{0}-J_{t}^{0}=1\right\vert \mathcal{F}_{t}^{\widetilde{Y}}\right] 
\text{ \ and \ }\lambda \widetilde{q}_t dt=\mathbb{\widetilde{P}}\left[ \left.
J_{t+dt}^{1}-J_{t}^{1}=1\right\vert \mathcal{F}_{t}^{\widetilde{Y}}\right] .
\end{equation*}%

In other words, the belief process will jump immediately to state $1$ (resp.
state $0$)\ when the Poisson signal arrives with the true information $%
\Theta =h$ (resp. $\Theta =l)$. We introduce an auxiliary constant, which will be used frequently throughout this section,

\begin{equation}
\widetilde{l}:=\frac{\rho }{\rho +\lambda }l+\frac{\lambda }{\rho +\lambda }%
(\mu -R).  \label{l-revised}
\end{equation}

\begin{proposition}
\label{prop:VB-poisson}
    The nested continuation value function $U^B(%
\widetilde{q}),$ $\widetilde{q}\in \left[ 0,1\right] $, is the unique continuous
solution to the variational inequality%
\begin{equation}
{\min }\Bigg( \rho U^B(\widetilde{q})-\lambda \Big(\widetilde{q}h+(1-\widetilde{q})(\mu
-R)-U^B(\widetilde{q})\Big)-\rho \Big( \widetilde{q}h+(1-\widetilde{q})l\Big)
,\, U^B(\widetilde{q})-\left( \mu -R\right) \Bigg) =0,  \label{HJB-Poisson}
\end{equation}%
with boundary conditions $U^B(0)=\mu -R$ and $U^B(1)=h$. It is given by 
\begin{equation}
U^B(\widetilde{q})=%
\begin{cases}
\mu -R,\quad  & \widetilde{q}\leq q_{B}, \\ 
\widetilde{q}h+(1-\widetilde{q})\widetilde{l},\quad  & \widetilde{q}>q_{B},%
\end{cases}%
\text{ }  \label{V-B}
\end{equation}%
with $\widetilde{l}$ as in (\ref{l-revised}) and 
\begin{equation*}
q_{B}=\rho \frac{\mu -l -R}{\lambda (\left( h-\mu \right) +R)+\rho (h-l)}.
\end{equation*}%
Furthermore, 
\begin{equation}
\lim_{\lambda \downarrow 0}q_{B}=\frac{ \mu -l -R}{h-l}\text{
\ and \ \ }\lim_{\lambda \uparrow \infty }q_{B}=0\text{,}  \label{lambda=limits}
\end{equation}%
and%
\begin{equation}
\widetilde{l}>l\text{ \ \ and \ \ }q_{B}<\hat{p},  \label{lambda-inequality}
\end{equation}%
with $\hat{p}$ as in \eqref{eq:p}.
\end{proposition}

\begin{proof}
The boundary conditions $U^B(0)=\mu -R$ and $U^B(1)=h$
follow directly. The form of equation (\ref{HJB-Poisson}) also follows,
observing that 
\begin{equation*}
\lambda \widetilde{q}\left( U^B(1)-U^B(\widetilde{q})\right) +\lambda (1-%
\widetilde{q})\left( U^B(0)-U^B(\widetilde{q})\right) =\lambda (%
\widetilde{q}h+(1-\widetilde{q})(\mu -R)-U^B(\widetilde{q})).
\end{equation*}%
Assertions (\ref{lambda=limits})\ and (\ref{lambda-inequality}) are obvious. \end{proof}

Given $U^B$ as established in Proposition \ref{prop:VB-poisson}, we define the value function $V$ (cf. (\ref{V-general})) with Poisson-type signals, by
\begin{equation}
V(q)=\sup_{\tau \in \mathcal{T}^{Y}_0}\mathbb{E}\left[ \left. -\int_{0}^{\tau
}e^{-\rho t}C(q_{t})dt+e^{-\rho \tau }V_1(q_{\tau })\,\right\vert
q_{0}=q\right] ,  \label{V-Poisson}
\end{equation}%
where $V_1(q)=\max \left( \mu ,\,V_1^B(q)\right)$, 
with $\,V_1^B(q) = U^B(\widetilde{q})$ as in (\ref{V-B}).
Note that because of the structure of $U^B$ as in (\ref{V-B}), the nested value function $V_1(q)$ has a piecewise linear form analogous to the obstacle in the irreversible decision problem, where the payoff
\(
G(q)=\max\{\mu,\; qh+(1-q)l\}.
\)
The only difference is that the constant $l$, representing the low payoff of
product~$B$ in the irreversible case, is replaced by
$\widetilde l$ defined in \eqref{l-revised}.
As a consequence, the analytical results established in
Proposition~\ref{prop:constant-2} apply directly in the present setting. See Figure \ref{fig:poisson} for illustration. 
\begin{figure}[H]
    \centering
    \includegraphics[width=13cm]{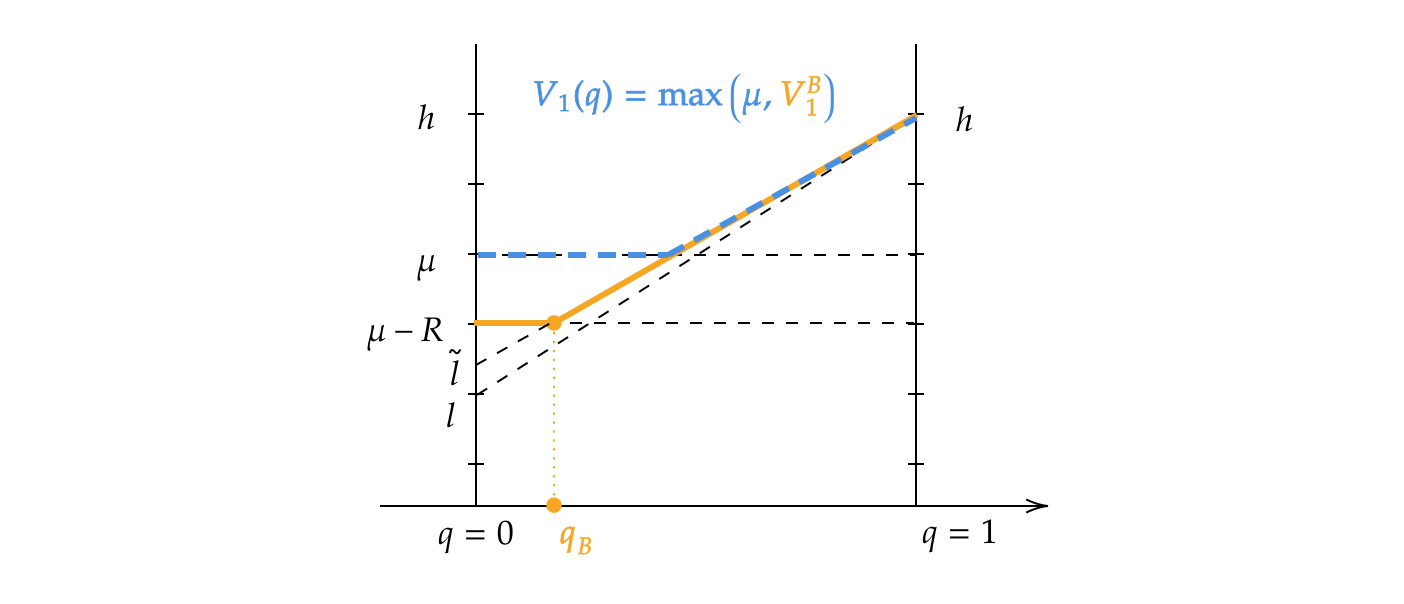}
    \caption{$V_1^B$ defined in \eqref{V-B} (in orange) v.s. $V_1$ defined in \eqref{G-general} (in dashed blue).}
    \label{fig:poisson}
\end{figure}
We stress, however,that while problems (\ref{V-Poisson}) and (\ref{V-constant})\ appear technically similar, they are generated by very different decision making models.

\subsubsection{Constant information cost}

To derive various properties of the value function $V(q), \ q \in[0,1]$ and compare them
to their analogues in the irreversible case that has been so far studied, we set $%
C(q)=C_{I}>0$. Then, 
\eqref{V-general} becomes
\begin{equation}
{\min }\left( \rho V(q)-\frac{1}{2}\left( \frac{h-l}{\sigma }\right)
^{2}q^{2}(1-q)^{2}V^{\prime \prime }\left( q\right) +C_{I},V(q)-\max \Big(
\mu ,\,V_1^B(q)\Big) \right) =0,  \label{HJB-Poisson-integrated}%
\end{equation}
with $V_1^B$ as in \eqref{V-B}.

\begin{theorem}%
\label{thm:regret}
Let the value function $V$ be as in \eqref{V-Poisson}. There exists a unique pair $(\underline{q},\overline{q}),$
with $0<\underline{q}<\overline{q}<1,$ such that $V(q)\in 
\mathcal{C}^{1}([0,1])$ and satisfies 
\begin{equation*}
V(q)=%
\begin{cases}
\mu ,\quad  & q\leq \underline{q}, \\ 
d_{1}q^{\frac{1-k}{2}}(1-q)^{\frac{1+k}{2}}+d_{2}(1-q)^{%
\frac{1-k}{2}}q^{\frac{1+k}{2}}{-}\frac{C_{I}}{\rho }%
,\quad  & \underline{q}\leq q\leq \overline{q}, \\ 
qh+(1-q)\widetilde{l},\quad  & \overline{q}\leq q,%
\end{cases}%
\end{equation*}%
with $\widetilde{l}$ defined in \eqref{l-revised}, $k$ as in \eqref{eqn:k}, 
\begin{equation*}
d_{1}=\frac{\frac{1+k}{2}-\underline{q}}{k(1-\underline{q})^{\frac{1+k}{2}}\underline{q}^{\frac{1-k}{2}}}\left( { \mu} +%
\frac{C_{I}}{\rho }\right) \text{\  \ and } \ \ 
d_{2}=-\frac{\frac{1-k}{2}-\underline{q} }{k(1-\underline{q})^{\frac{1-k}{2}}%
\underline{q}^{\frac{1+k}{2}}}\left( { \mu}+\frac{C_{I}}{\rho }%
\right).
\end{equation*}%
\end{theorem}

\begin{proof}
 The ODE%
\begin{equation}
\label{eqn:constant:sol-onlyPoisson}
\rho V-\frac{(h-l)^{2}}{2\sigma ^{2}}q^{2}(1-q)^{2}V^{\prime \prime }+C_{I}=0
\end{equation}%
has a general solution given by 
\begin{equation}\label{eqn:ode_general_solution}
V(q)=d_{1}v_{1}(q)+d_{2}v_{2}(q)-\frac{C_{I}}{\rho },
\end{equation}%
where $d_{1},d_{2}$ are parameters to be determined, and 
\begin{equation*}
v_{1}(q)=q^{\frac{1-k}{2}}(1-q)^{\frac{1+k}{2}}\quad \text{and \ \ }
v_{2}(q)=q^{\frac{1+k}{2}}(1-q)^{\frac{1-k}{2}},
\end{equation*}%
with $k$ as in \eqref{eqn:k}. Working as in the proof of Theorem \ref{thm:no-regret}, 
we deduce that there exist points $\underline{q},\overline{q}$ with $0<%
\underline{q}<\overline{q}<1$, such that $V(q)=\mu ,$ for $q\in \lbrack 0,%
\underline{q}]$; $V(q)=qh+(1-q)\,\widetilde{l},$ for $q\in \lbrack \overline{%
q},1]$, and that $V(q)$ satisfies the ODE \eqref{eqn:constant:sol-onlyPoisson}
for $q\in \lbrack \underline{q},\overline{q}]$. Therefore, by the smooth-fit
principle, $d_{1},d_{2},\underline{q},\overline{q}$ must satisfy

\begin{equation*}
\begin{cases}
d_{1}v_{1}(\underline{q})+d_{2}v_{2}(\underline{q})-\frac{C_{I}}{\rho } & 
={ \mu}, \\ 
d_{1}v_{1}^{\prime }(\underline{q})+d_{2}v_{2}^{\prime }(\underline{q}) & =0,
\\ 
d_{1}v_{1}(\overline{q})+d_{2}v_{2}(\overline{q})-\frac{C_{I}}{\rho } & =%
\overline{q}h+(1-\overline{q})\,\widetilde{l}, \\ 
d_{1}v_{1}^{\prime }(\overline{q})+d_{2}v_{2}^{\prime }(\overline{q}) & =h-\,%
\widetilde{l}.%
\end{cases}%
\end{equation*}%
Technically, two of the four equations above are sufficient to write $d_1,d_2$ in terms of $\overline{q},\underline{q}$. For simplicity, we write both $d_1$ and $d_2$ in terms of $\underline{q}$,  which utilizes the first two equations. To simplify the notation, denote $m:=\frac{1-k}{2}$. Then, the general
solution  \eqref{eqn:ode_general_solution} can be written as 
\begin{equation*}
V(q)=d_{1}q^{m}(1-q)^{-m+1}+d_{2}(1-q)^{m}q^{-m+1}-\frac{C_{I}}{\rho }.
\end{equation*}%
In turn, 
\begin{eqnarray*}
V^{\prime }(q) &=&d_{1}mq^{m-1}(1-q)^{-m+1}+d_{1}(m-1)q^{m}(1-q)^{-m} \\
&&+d_{2}(-m)(1-q)^{m-1}q^{-m+1}+d_{2}(-m+1)(1-q)^{m}q^{-m} \\
&=&d_{1}q^{m-1}(1-q)^{-m}\left( m(1-q)+(m-1)q\right)
+d_{2}(1-q)^{m-1}q^{-m}\left( -mq+(-m+1)(1-q)\right)  \\
&=&d_{1}q^{m-1}(1-q)^{-m}(m-q)+d_{2}(1-q)^{m-1}q^{-m}\left( -m+1-q\right) .
\end{eqnarray*}%
At point $\underline{q}$, we must have
\begin{eqnarray}
&&d_{1}\underline{q}^{m}(1-\underline{q})^{-m+1}+d_{2}(1-\underline{q})^{m}%
\underline{q}^{-m+1}={ \mu} +\frac{C_{I}}{\rho },  \label{eq:poisson:left1} \\
&&d_{1}(m-\underline{q})\underline{q}^{m-1}(1-\underline{q}%
)^{-m}+d_{2}\left( -m+1-\underline{q}\right) (1-\underline{q})^{m-1}%
\underline{q}^{-m}=0.  \label{eq:poisson:left2}
\end{eqnarray}%
From \eqref{eq:poisson:left2}, we then obtain
\begin{equation}
d_{1}=-d_{2}\frac{-m+1-\underline{q}}{m-\underline{q}}\frac{(1-\underline{q}%
)^{m-1}\underline{q}^{-m}}{\underline{q}^{m-1}(1-\underline{q})^{-m}}=-d_{2}%
\frac{-m+1-\underline{q}}{m-\underline{q}}\frac{(1-\underline{q})^{2m-1}}{%
\underline{q}^{2m-1}}.
\label{eq:poisson:left3}
\end{equation}%
Plugging the above equation into \eqref{eq:poisson:left1} gives 
\begin{equation*}
-d_{2}\frac{-m+1-\underline{q}}{m-\underline{q}}\frac{(1-\underline{q}%
)^{2m-1}}{\underline{q}^{2m-1}}\underline{q}^{m}(1-\underline{q}%
)^{-m+1}+d_{2}(1-\underline{q})^{m}\underline{q}^{-m+1}={ \mu} +\frac{C_{I}}{%
\rho }.
\end{equation*}%
By direct computation, we deduce 
\begin{equation*}
-d_{2}\frac{-m+1-\underline{q}}{m-\underline{q}}(1-\underline{q}%
)^{m}q^{-m+1}+d_{2}(1-\underline{q})^{m}(\underline{q})^{-m+1}={ \mu} +\frac{%
C_{I}}{\rho }.
\end{equation*}%
Hence $d_{2}\left( 1-\frac{-m+1-\underline{q}}{m-\underline{q}}\right) (1-%
\underline{q})^{m}\underline{q}^{-m+1}={ \mu} +\frac{C_{I}}{\rho }$ and, in
turn, 
\begin{equation}
d_{2}=\left( { \mu} +\frac{C_{I}}{\rho }\right) \frac{m-\underline{q}}{2m-1}%
(1-\underline{q})^{-m}\underline{q}^{m-1}.
\label{eq:poisson:c2_q1}
\end{equation}%
Plugging \eqref{eq:poisson:c2_q1} into \eqref{eq:poisson:left3}, gives 
\begin{equation*}
d_{1}=-\left( { \mu} +\frac{C_{I}}{\rho }\right) \frac{-m+1-\underline{q}}{%
2m-1}(1-\underline{q})^{m-1}\underline{q}^{-m}.
\end{equation*}
\end{proof}

Next, we investigate how the exchange fee $R$ affects the actions of the
DM\ in the first exploration period $[0,\tau^\ast]$. 

\newcommand{\tl}{\widetilde{l}}
\newcommand{\tG}{V_1}

\begin{proposition}
\label{prop:constant:cs}
Let the model input values of $l$, $\mu $, $h$, $\rho $, $C_{I}$, $%
\sigma $, and $\lambda $ be fixed. If the exchange fee $R$ increases, then both $%
\underline{q}$ and $\overline{q}$ increase. 
\end{proposition}

\begin{proof}
     Let $%
R_{2}<R_{1}$ and denote $\widetilde{l}_{1},\widetilde{l}_{2}$ as in  \eqref{l-revised} with $R_1$,$R_2$ respectively. Then, $\widetilde{l}_{1}<\widetilde{l}_{2}$. Let 
\begin{equation*}
V_1^{R_1}(q)=\max \left( \mu ,qh+(1-q)\widetilde{l}_{1}\right) \text{
\ \ and \ \ }V_1^{R_2}(q)=\max \left( \mu ,qh+(1-q)\widetilde{l}_{2}\right) ,
\end{equation*}%
and note that, for $q\in \left[ 0,1\right] ,$  we have 
$$\frac{h-\widetilde{l}_{2}}{h-\widetilde{l}_{1}}V_1^{R_1}(q)+\frac{%
\widetilde{l}_{2}-\widetilde{l}_{1}}{h-\widetilde{l}_{1}}h\,\geq\, V_1^{R_2}(q)\,\geq\, V_1^{R_1}(q).$$

Let $V^{R_1}$ and $V^{R_2}$ be the respective viscosity solutions to 
\begin{align}
\min \left( \rho V^{R_1}(q)-\frac{1}{2}\left( \frac{h-l}{\sigma }\right)
^{2}q^{2}(1-q)^{2}(V^{R_1})^{\prime \prime }(q)+C_{I},V^{R_1}(q)-V_1^{R_1}(q)\right) & =0,  \label{eqn:hjb-cs1} \\
\min \left( \rho V^{R_2}(q)-\frac{1}{2}\left( \frac{h-l}{\sigma }\right)
^{2}q^{2}(1-q)^{2}(V^{R_2})^{\prime \prime }(q)+C_{I},V^{R_2}(q)-V_1^{R_2}(q)\right) & =0. \label{eqn:hjb-cs2}
\end{align}%
Using that $V_1^{R_1}(q)\leq V_1^{R_2}(q)$ and working as in
the proof of Proposition \ref{prop:constant-2}, we deduce that $V^{R_1}$ is a viscosity
subsolution to \eqref{eqn:hjb-cs2}. By the comparison principle for continuous viscosity solutions, we conclude that $%
V^{R_1}\left( q\right) \leq V^{R_2}\left( q\right) $.

Next, we show that the function $\frac{h-\widetilde{l}_{2}}{h-\widetilde{l}%
_{1}}V^{R_1}+\frac{\widetilde{l}_{2}-\widetilde{l}_{1}}{h-\widetilde{l}_{1}}h$
is a viscosity supersolution to \eqref{eqn:hjb-cs2}. For this, we take $%
q_0\in (0,1)$ and a test function $\psi \in \mathcal C^{2}((0,1))$, such that 
\begin{equation*}
\left( \frac{h-\widetilde{l}_{2}}{h-\widetilde{l}_{1}}V^{R_1}+\frac{\widetilde{%
l}_{2}-\widetilde{l}_{1}}{h-\widetilde{l}_{1}}h-\psi \right)
(q_0)=\min_{q\in (0,1)}\left( \frac{h-\widetilde{l}_{2}}{h-\widetilde{l}%
_{1}}V^{R_1}\left( q\right) +\frac{\widetilde{l}_{2}-\widetilde{l}_{1}}{h-%
\widetilde{l}_{1}}h\left( q\right) -\psi \left( q\right) \right) =0.
\end{equation*}%
Then, the test function $\varphi :=\frac{h-\widetilde{l}_{1}}{h-\widetilde{l}%
_{2}}\left( \psi -\frac{\widetilde{l}_{2}-\widetilde{l}_{1}}{h-\widetilde{l}%
_{1}}h\right) $ satisfies $\varphi \in \mathcal{C}^{2}((0,1))$ and 
\begin{equation*}
(V^{R_1}-\varphi )(q_0)=\min_{q\in (0,1)}(V^{R_1}\left( q\right) -\varphi
\left( q\right) )=0.
\end{equation*}%
Since $V^{R_1}$ is the viscosity solution to \eqref{eqn:hjb-cs1}, it is a
viscosity supersolution to \eqref{eqn:hjb-cs1} and, thus, 
\begin{equation*}
\min \left( \rho \varphi (q_0)-\frac{1}{2}\left( \frac{h-l}{\sigma }%
\right) ^{2}q_0^{2}(1-q_0)^{2}\varphi^{\prime \prime
}(q_0)+C_{I},\varphi (q_0)-V_1^{R_1}(q_0)\right) \geq 0.
\end{equation*}%
On the other hand, $\varphi \leq \frac{h-\widetilde{l}_{1}}{h-\widetilde{l}_{2}}\psi $, $\varphi ^{\prime \prime }=\frac{h-%
\widetilde{l}_{1}}{h-\widetilde{l}_{2}}\psi ^{\prime \prime }$, $C_{I}\leq\frac{h-\widetilde{l}_{1}}{h-\widetilde{l}_{2}}C_{I}$, and $$\varphi -V_1^{R_1}=\frac{h-\widetilde{l}_{1}%
}{h-\widetilde{l}_{2}}\psi -\frac{\widetilde{l}_{2}-\widetilde{l}_{1}}{h-%
\widetilde{l}_{2}}h-V_1^{R_1}\leq \frac{h-\widetilde{l}_{1}}{h-%
\widetilde{l}_{2}}(\psi -V_1^{R_2}),$$ and, thus, 
\begin{equation*}
\min \left( \rho \psi (q_0)-\frac{1}{2}\left( \frac{h-l}{\sigma }\right)
^{2}q_0^{2}(1-q_0)^{2}\psi ^{\prime \prime }(q_0)+C_{I},\psi
(q_0)-V_1^{R_2}(q_0)\right) \geq 0.
\end{equation*}%
Therefore, function $\frac{h-\widetilde{l}_{2}}{h-\widetilde{l}_{1}}V^{R_1}+\frac{%
\widetilde{l}_{2}-\widetilde{l}_{1}}{h-\widetilde{l}_{1}}h$ is a viscosity
supersolution to \eqref{eqn:hjb-cs2}. In turn, the comparison principle of the continuous viscosity solutions
yields that it dominates $V^{R_2}.$ 

In summary, we have so far shown that 
\begin{equation*}
\frac{h-\widetilde{l}_{2}}{h-\widetilde{l}_{1}}V^{R_1}(q)+\frac{\widetilde{l}%
_{2}-\widetilde{l}_{1}}{h-\widetilde{l}_{1}}h\geq V^{R_2}(q)\geq V^{R_1}(q),%
\text{ \ \ }q\in \lbrack 0,1].
\end{equation*}%
From Theorem \ref{thm:no-regret}, there exist two pairs of cutoff points $%
\left( \underline{q}_{1},\overline{q}_{1}\right) $ $\ $and $\ \left( 
\underline{q}_{2},\overline{q}_{2}\right) $ such that 
\begin{equation*}
\begin{cases}
V^{R_1}(q)=\mu,  & q\in \lbrack 0,\underline{q}_{1}], \\ 
V^{R_1}(q)>V_1^{R_1}(q), & q\in (\underline{q}_{1},\overline{q}_{1}), \\ 
V^{R_1}(q)=hq+\widetilde{l}_{1}(1-q), & q\in \lbrack \overline{q}_{1},1],\quad 
\end{cases}%
\qquad 
\begin{cases}
V^{R_2}(q)=\mu,  & q\in \lbrack 0,\underline{q}_{2}], \\ 
V^{R_2}(q)>V_1^{R_2}(q), & q\in (\underline{q}_{2},\overline{q}_{2}), \\ 
V^{R_2}(q)=hq+\widetilde{l}_{2}(1-q), & q\in \lbrack \overline{q}_{2},1].%
\end{cases}%
\end{equation*}%
To compare the points $\underline{q}_{1}$ and $\underline{q}_{2}$, we first note that $%
V^{R_2}(q)\geq V^{R_1}(q)\geq \mu .$ For any $q\in \lbrack 0,1]$ such that $%
V^{R_2}(q)=\mu $, we also have $V^{R_1}(q)=\mu $ and, hence, 
\begin{equation*}
\lbrack 0,\underline{q}_{2}]=\{q\in \lbrack 0,1]:V^{R_2}(q)=\mu \}\subseteq
\{q\in \lbrack 0,1]:V^{R_1}(q)=\mu \}=[0,\underline{q}_{1}],
\end{equation*}%
which implies that $\underline{q}_{2}\leq \underline{q}_{1}$. 

Analogously,
to compare $\overline{q}_{1},\overline{q}_{2}$, we observe that 
\begin{equation*}
\frac{h-\widetilde{l}_{2}}{h-\widetilde{l}_{1}}V^{R_1}+\frac{\widetilde{l}_{2}-%
\widetilde{l}_{1}}{h-\widetilde{l}_{1}}h\geq \frac{h-\widetilde{l}_{2}}{h-%
\widetilde{l}_{1}}(hq+\widetilde{l}_{1}(1-q))+\frac{\widetilde{l}_{2}-%
\widetilde{l}_{1}}{h-\widetilde{l}_{1}}h=hq+\widetilde{l}_{2}(1-q).
\end{equation*}%
For any $q\in \lbrack 0,1]$ such that $V^{R_1}(q)=hq+\widetilde{l}_{1}(1-q)$,
we also have $V^{R_2}(q)=hq+\widetilde{l}_{2}(1-q)$ and, therefore, 
\begin{equation*}
\lbrack \overline{q}_{1},1]=\Big\{q\in \lbrack 0,1]:V^{R_1}(q)=hq+\widetilde{l}%
_{1}(1-q)\Big\}
\subseteq \Big\{q\in \lbrack 0,1]:V^{R_2}(q)=hq+\widetilde{l}_{2}(1-q)\Big\}=[%
\overline{q}_{2},1],
\end{equation*}%
which yields $\overline{q}_{2}\leq \overline{q}_{1}$.

In conclusion, for $R_{2}<R_{1}$, or equivalently $\widetilde{l}_{2}>\widetilde{l}_{1}$, we have that $\underline{q}_{2}\leq \underline{q}_{1}$
and $\overline{q}_{2}\leq \overline{q}_{1}$. 
\end{proof}

Proposition \ref{prop:constant:cs} implies that, when the exchange fee $R$
increases, the DM becomes more\textit{\ reluctant} to leave the safe choice
region (since $\underline{q}$ increases) and take the risk to choose product 
$B$ (since $\overline{q}$ increases).

\paragraph{Comparison with the single (irreversible) decision problems.}

As mentioned earlier, $V_1(q)$ has the same structure as $G(q)$ (cf. \eqref{G-linear}), by replacing $l$ by $\widetilde{l}$. Hence the solution of the
extended model with a Poisson-type refined signal may be viewed as a
single, irreversible problem with a modified low value for the unknown
product. Conceptually, however, we are dealing with entirely different models. 

The quantity $\widetilde{l}$ decreases when  $R$ increases,
with the limiting value $\widetilde{l}=l$ achieved when $R=\mu -l$. By
rewriting $\widetilde{l}=-\frac{\rho }{\rho +\lambda }(\mu -R-l)+(\mu -R)$,
we also see that $\widetilde{l}$ increases when $\lambda $ increases.

Proposition \ref{prop:constant-2} suggests that,
asymptotically, the bandwidth of the exploration region $[\underline{q},%
\overline{q}]$ converges to zero as $\lambda \uparrow \infty $ and $%
R\downarrow 0$. This implies that when the Poisson signal arrives fast
enough and when the exchange fee is low, the DM spends less time
acquiring information in the first exploratory period, and  makes
her initial choice comparatively sooner than in the irreversible case.

\subsection{Case 2: Sequential information signal of Gaussian type of higher accuracy}\label{sec:higher accuracy}

The new observation process $\widetilde{Y}$ is modeled as 
\begin{equation}
\label{eq:signal-small-sigma}
d\widetilde{Y}_{t}=\Theta \,dt+\widetilde{\sigma }d\widetilde{W}_{t},
\end{equation}%
with $0<\widetilde{\sigma }\leq \sigma ,$ $\sigma $ being the volatility
coefficient of the signal $Y$ used in the first exploratory period, and $\widetilde{W}_t$ is a standard Brownian motion independent of $W _t$. In
analogy to
\eqref{eq:belief-small-sigma}, the related belief process $\widetilde{q}$ satisfies 
\begin{equation}\label{eq:belief-small-sigma}
d\widetilde{q}_{t}=\frac{h-l}{\widetilde{\sigma }}\widetilde{q}_{t}\left( 1-%
\widetilde{q}_{t}\right) d\widetilde{Z}_{t},\text{ \ }\widetilde{q}_{0}=\widetilde{q}\in %
\left[ 0,1\right],
\end{equation}%
with $\widetilde{Z}$ being a standard Brownian motion with respect to the filtration
$\{\mathcal F_t^{\widetilde{Y}}\}_{t\ge 0}$ on the probability space $%
\left( \Omega ,\mathcal{F},\mathbb{P}%
\right)$.
The corresponding variational inequality for the nested continuation value
function of product $B$ is 
\begin{equation*}
{\min }\left( \rho U^B(\widetilde{q})-\frac{1}{2}\left( \frac{h-l}{\widetilde{\sigma }}\right) ^{2}%
\widetilde{q}^{2}(1-\widetilde{q})^{2}(U^B)^{\prime \prime }(\widetilde{q})-\rho \left( \widetilde{q}h+(1-\widetilde{q}%
)l\right) ,U^B(\widetilde{q})-\left(
\mu -R\right) \right) =0,
\end{equation*}%
with $U^B(0)=\mu -R$ and $U^B(1)=h.$

\begin{proposition}   
\label{thm:small-sigma-V1} The nested continuation
value function for product $B{\ }$is given by 
\begin{equation}
U^B(\widetilde{q})=%
\begin{cases}
\mu -R,\quad  & \widetilde{q}\leq q_B, \\ 
\widetilde{q}h+(1-\widetilde{q})l+{d}_{B}\widetilde{q}^{\frac{1-\widetilde k}{2}}(1-\widetilde{q})^{\frac{1+\widetilde{k}}{2}%
},\quad  & \widetilde{q}>q_B,%
\end{cases}
\label{V-B-Gaussian}
\end{equation}%
where 
\begin{equation*}
q_B=\frac{\left( l-\mu +R\right) \left( \frac{1-\widetilde{k}}{2}\right) }{(h-l)\left( \frac{1+\widetilde k}{2}%
\right) +l-\mu +R},
\end{equation*}%
with 
\begin{equation}
{d}_{B}=\frac{\mu-R-q_Bh-(1-q_B)l}{%
q_B{}^{\frac{1-\widetilde k}{2}}(1-\widetilde{q}%
_{B})^{\frac{1+\widetilde k}{2}}}\text{ \ \ and \ \ }\widetilde{k}%
=\sqrt{1+8\rho \left( \frac{\widetilde{\sigma }}{h-l}\right) ^{2}}.
\label{d(b)-k}
\end{equation}
\end{proposition}

\begin{proof}
We first note that the solution of the ODE (with a slight abuse of notation) 
\begin{equation}
\rho U^B (\widetilde q)-\frac{1}{2}%
\left( \frac{h-l}{\widetilde{\sigma }}\right) ^{2}\widetilde{q}^{2}(1-%
\widetilde{q})^{2}(U^B)^{\prime \prime }(\widetilde{q})-\rho \Big( \widetilde{q}h+(1-\widetilde{q})l\Big) =0
\label{eq:ode:small}
\end{equation}%
can be written as 
\begin{equation*}
U^B(\widetilde{q})=\widetilde{q}h+(1-\widetilde{q})l+d_B v_{1}(\widetilde{%
q})+d_A v_{2}(\widetilde{q}),
\end{equation*}
where $v_{1}(\widetilde{q})=\widetilde q^{\frac{1-\widetilde k}{2}}(1-\widetilde{q}%
)^{\frac{1+\widetilde k}{2}}$ and,  $v_{2}(%
\widetilde{q})=\widetilde{q}^{\frac{1+\widetilde k}{2}}(1-%
\widetilde{q})^{\frac{1-\widetilde k}{2}}$, with
$d_B$ and $d_A$ two free parameters to be determined and $\widetilde{k}$
as in (\ref{d(b)-k}).

Working as in the proof of Proposition \ref{prop:VB-poisson}, we seek $U^B\in 
\mathcal{C}^{1}([0,1])$ and $q_B\in (0,1)$ such that $U^B(\widetilde q)=\mu -R,$  $q\leq \widetilde q_B$, and $U^B(\widetilde{q}%
)$ satisfies ODE \eqref{eq:ode:small} for $\widetilde{q}>q_B$. 

Since $U^B$ is bounded near $\widetilde{q}=1$ we conclude that $d_A=0$, otherwise we would have $\lim_{\widetilde{q}\rightarrow 1}|U^B(\widetilde{%
q})|=+\infty $ since $\lim_{\widetilde{q}\rightarrow 1}v_{2}(%
\widetilde{q})=\infty $. 

Applying the smooth fit principle yields 
\begin{eqnarray*}
 q_Bh+(1-q_B)l+d_B%
q_B{}^{\frac{1-\widetilde k}{2}}(1-{q}%
_{B})^{\frac{1+\widetilde k}{2}}=\mu -R,\,\,
 h-l+d_Bq_B{}^{-\frac{1+\widetilde{k}
}{2}}(1-q_B)^{-\frac{1-\widetilde{k}}{2}}\left( \frac{1-\widetilde k}{2}-{q}%
_{B}\right) =0,  
\end{eqnarray*}%
which, in turn, implies 
\begin{eqnarray*}
q_B =\frac{\left(l-\mu+R\right) \left( \frac{1-\widetilde{k}}{2}\right) }{(h-l)\left( \frac{1+\widetilde k}{2}%
\right) +l-\mu+R} \quad \text{ and \ \ \ }
{d}_{B} =\frac{\mu-R-q_Bh-(1-q_B)l}{%
q_B{}^{\frac{1-\widetilde k}{2}}(1-{q}%
_{B})^{\frac{1+\widetilde k}{2}}}.  \label{eq:oppo:small-sigma-c1}
\end{eqnarray*}
\end{proof}
\vspace{-15pt}
\begin{lemma}\label{lemma:g}
 If the exchange fee $R>0$, there exists a unique $%
q'\in (0,1)$ such that $U^B(q')=\mu $.
\end{lemma}
\begin{proof}
Since $U^B(q_B)=\mu -R<\mu $ and $%
U^B(1)=h>\mu $, the continuity of $U^B$ implies that there exists $%
\widetilde{q}\in (q_B,1)$ such that $U^B(\widetilde{q})=\mu $. Moreover, using the form of $U^B$, we can show that for $\widetilde{q}\in (%
q_B,1)$, we have 
\begin{equation*}
(U^B)^{\prime \prime }(\widetilde{q})=\frac{d_{B}(\widetilde{k}^{2}-1)}{4 \widetilde{q}^{\frac{3+\widetilde{k}}{2}}(1-\widetilde{q})^{\frac{3-\widetilde{k}}{2}}}>0.
\end{equation*}%
Therefore, $U^B(\widetilde{q})$ is convex for $\widetilde{q}\in (%
q_B,1)$. 

Next, we assume that there exist at least two points $q_{1}$
and $q_{2}$ such that $U^B(q_{1})=U^B(q_{2})=\mu $, with $\widetilde{q}%
_{B}<q_{1}<q_{2}<1$. Then, by the convexity of $U^B$, we  have 
\begin{equation*}
U^B(q_{1})\leq \frac{q_{2}-q_{1}}{q_{2}-q_B}U^B(\widetilde{%
q}_{B})+\frac{q_{1}-q_B}{q_{2}-q_B}%
U^B(q_{2})<\mu ,
\end{equation*}%
which, however, contradicts that $U^B(q_{1})=\mu $. Therefore, there must
exist a unique\textit{\ }point $q'$ such that $U^B(q') = \mu$. 
\end{proof}

Given $U^B$ as established in Proposition \ref{thm:small-sigma-V1}, we define the value function $V$ (cf. (\ref{V-general})) with Gaussian-type signals, by
\begin{equation}
V(q)=\sup_{\tau \in \mathcal{T}^{Y}_0}\mathbb{E}\left[ \left. -\int_{0}^{\tau
}e^{-\rho t}C(q_{t})dt+e^{-\rho \tau }V_1(q_{\tau })\,\right\vert
q_{0}=q\right] ,  \label{V-Gau}
\end{equation}%
where $V_1(q)=\max \left( \mu ,\,V_1^B(q)\right)$, 
with $\,V_1^B(q) = U^B(\widetilde{q})$ as in (\ref{V-B-Gaussian}). See Figure \ref%
{fig:g} for a demonstration.
\vspace{-15pt}
\begin{figure}[H]
    \centering
\includegraphics[width=15cm]{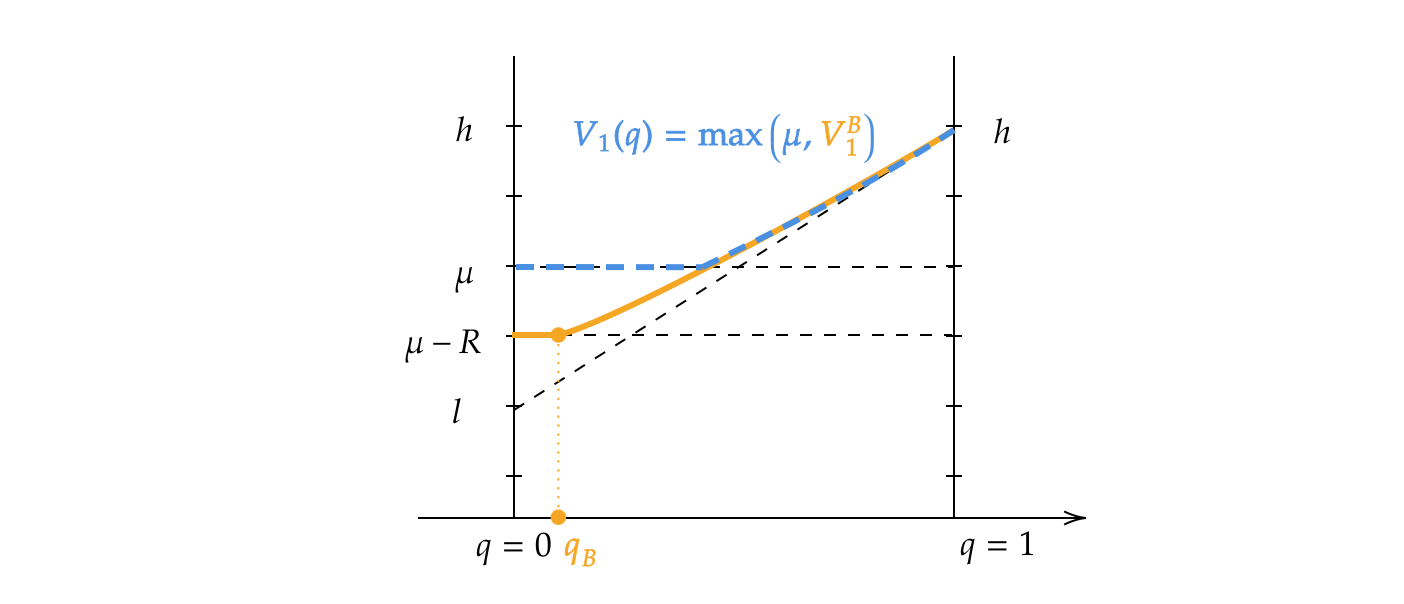}
    \caption{$V_1^B$ defined in \eqref{V-B-Gaussian} (in orange) v.s. $V_1$ defined in \eqref{G-general} (in dashed blue).}
    \label{fig:g}
\end{figure}

\subsubsection{Constant information cost}
We assume that $C(q)=C_{I}>0,$ $q\in \left[ 0,1\right] $. Then, the value
function $V$ satisfies   
\begin{equation*}
{\min }\left( \rho V(q)-\frac{1}{2}\left( \frac{h-l}{\sigma }\right)
^{2}q^{2}(1-q)^{2}V^{\prime \prime }\left( q\right) +C_{I},V(q)-V_1%
(q)\right) =0,\quad \quad V_1(q)=\max \Big( \mu ,V_1^B(q)\Big), 
\end{equation*}%
with $V_1^B (q)= U^B(\widetilde{q})$ given in (\ref{V-B-Gaussian}). Working as in the proof of Theorem \ref{thm:regret} we deduce the following result. The arguments are similar,
albeit more tedious but are omitted for brevity.

\begin{theorem}\label{thm:small:regret}
There exists a unique pair $(\underline{q},%
\overline{q})$, with $0<\underline{q}<\overline{q}<1$, such that the value
function $V(q)\in \mathcal{C}^{1}\left( \left[ 0,1\right] \right) $ and
satisfies 
\begin{equation*}
V(q)=
\begin{cases}
{ \mu},\quad  & q\leq \underline{q}, \\ 
d_{1}q^{\frac{1-k}{2}}(1-q)^{\frac{1+k}{2}}+d_{2}(1-q)^{%
\frac{1-k}{2}}q^{\frac{1+k}{2}}-\frac{C_{I}}{\rho }%
,\quad  & \underline{q}\leq q\leq \overline{q}, \\ 
qh+(1-q)l+d_{B}q^{\frac{1-\widetilde k}{2}}(1-q)^{\frac{1+\widetilde{k}}{2}},\quad  & \overline{q}\leq q,%
\end{cases}%
\end{equation*}%
where 
\begin{equation*}
d_{1}=\frac{\frac{1+k}{2}-\underline{q}}{k(1-\underline{q})^{\frac{1+k}{2}}\underline{q}^{\frac{1-k}{2}}}\left( { \mu} +%
\frac{C_{I}}{\rho }\right) \text{, \ }d_{2}=-\frac{\frac{1-k}{2}-\underline{q} }{k(1-\underline{q})^{\frac{1-k}{2}}%
\underline{q}^{\frac{1+k}{2}}}\left( { \mu}+\frac{C_{I}}{\rho }%
\right),
\end{equation*}%
and%
\begin{equation*}
d_{B}=\frac{\mu -l-R}{\frac{1+\widetilde k}{2}}\left( \frac{%
\left( \frac{1+\widetilde k}{2}\right) (h-\mu +R)}{-\left( \frac{1-\widetilde{k}}{2}\right) (\mu -l-R)}\right) ^{\frac{1-\widetilde{k}}{2}} \,\, \text{ with }\,\, \widetilde{k}=\sqrt{1+8\rho \left( \frac{\widetilde{\sigma}}{h-l}\right) ^{2}}.
\end{equation*}%

\end{theorem}

\subsubsection{Comparison with the solution of the single (irreversible)
decision problem}
Analytically, we can show, using Theorem %
\ref{thm:small-sigma-V1}, or following similar arguments as in Section \ref{sec:precise_information} that as $\widetilde{\sigma }\downarrow 0$, then $q_{B}\downarrow 0$ and $U^B(\widetilde{q}%
)=\widetilde{q}h+(1-\widetilde{q})\left( \mu -R\right) $. Thus, we may view $%
V_1(q)$ as the analogous $G(q)$ function in the irreversible
decision making case but with the lower value of product $B$ replaced by $\mu
-R$. This implies that, if the DM could observe immediately the true value
of product $B$ should she choose it, she would immediately finalize her
decision whether to keep it if $\Theta =h$, or switch to product $A$ at a
cost of $R$, if $\Theta =l$.

Also note that when $\widetilde{\sigma }=0$, the same monotonicity result
with respect to changes in $R$ in Proposition \ref{prop:constant:cs} holds,
i.e., as $R$ increases, both $\underline{q}$ and $\overline{q}$ increase.
This indicates that when the cost of revising the initial decision is
higher, the DM is more likely to choose the well-known product $A$.

Finally, we further conduct a series of numerical experiments when $0<%
\widetilde{\sigma }<\sigma $, as illustrated in Figure \ref%
{fig:small_variance_cs}, to further investigate the impact of the exchange
cost $R$ on the behavior of DM.

\begin{figure}[H]
    \centering
    \includegraphics[width=0.55\textwidth]{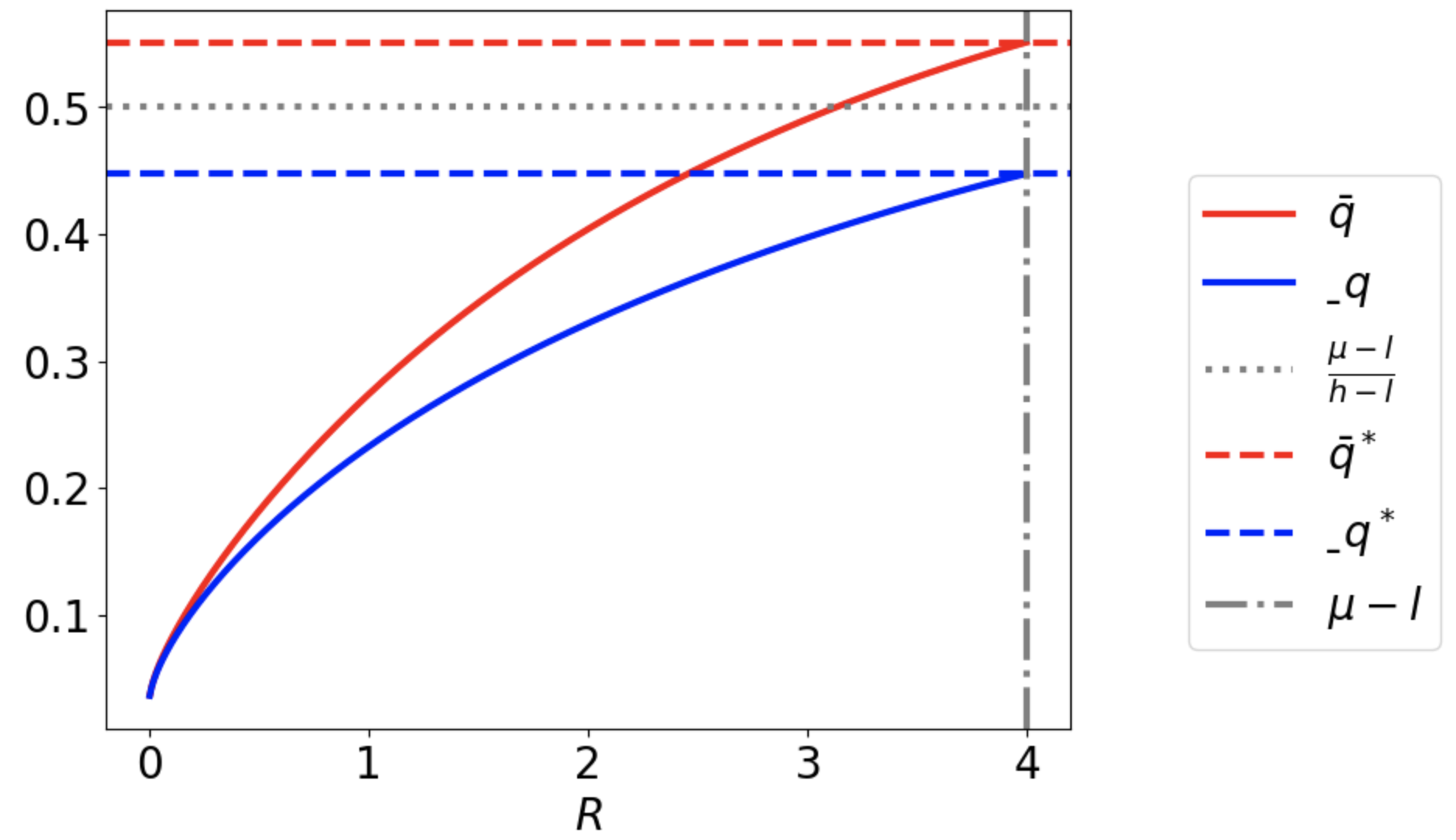}
    \includegraphics[width=0.43\textwidth]{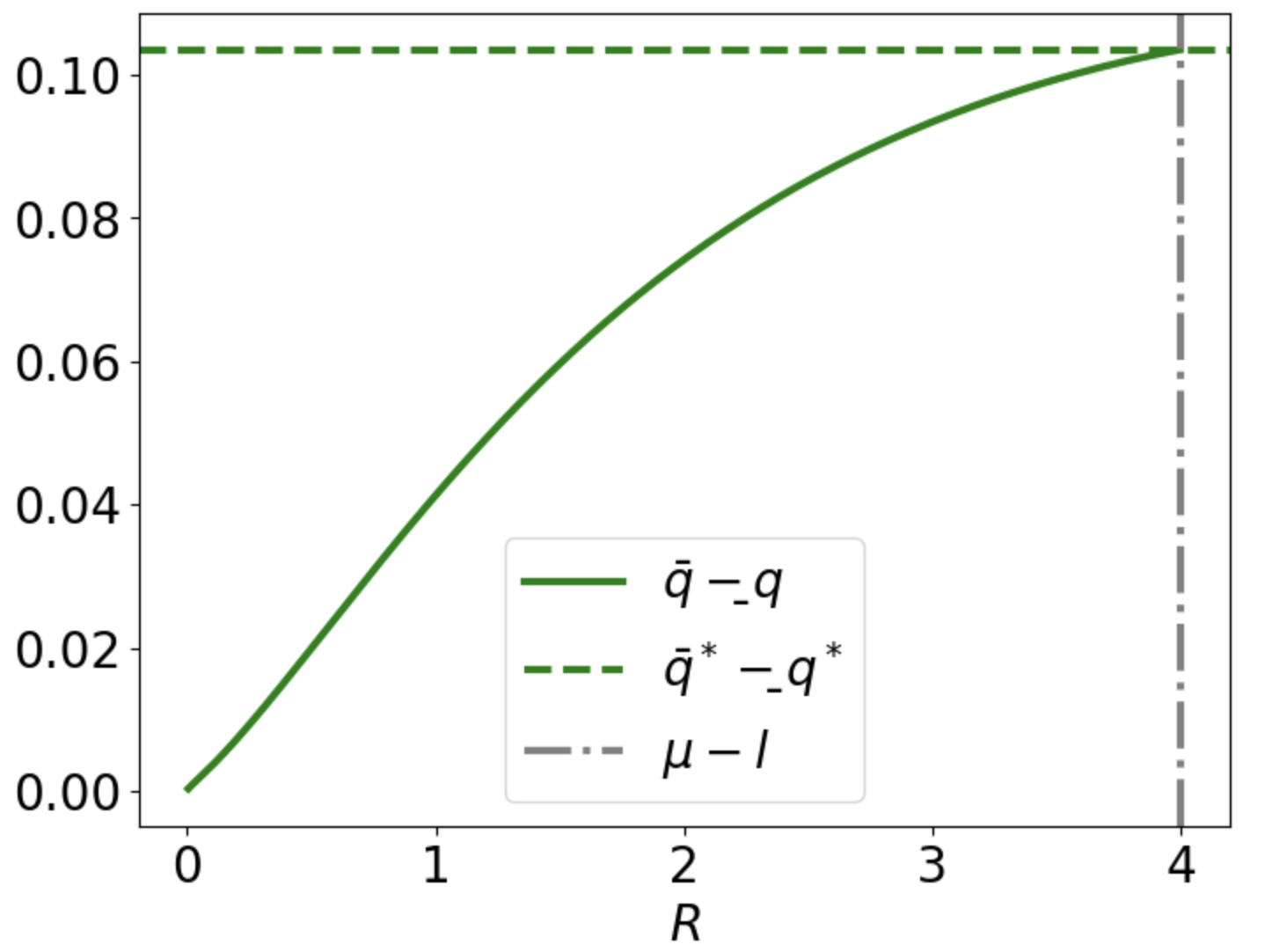}
    \caption{Behavior of $\qis$, $\qiis$, and $\qiis-\qi$. In this experiment, we take $\rho = 1$, $l=1$, $h=9$, $\mu = 5$, $\cI=1$, $\sigma = 5$, $\tsigma = 1$.}
    \label{fig:small_variance_cs}
\end{figure}
For demonstration purposes, we denote the cutoffs of the irreversible counterpart as $\underline{%
q}^{\ast }$ and $\overline{q}^{\ast }$. As the exchange fee $R$ increases to $\mu -l
$, we observe that:

\begin{itemize}
    \item[i)] 
Both points, $\underline{q}$ and $\overline{q},$ increase
monotonically and converge towards $\underline{q}^{\ast }$ and $\overline{q}%
^{\ast },$ respectively. This indicates that when exchanging the product becomes more
costly, the behavior of the DM tends to align more closely with the one
exhibited in the single (irreversible) decision problem.

\item[ii)] The width of the exploration region $\overline{q}-\underline{q}$ also
expands from zero  to  $\overline{q}^{\ast }-\underline{q}%
^{\ast }$, the exploration width of the corresponding single (irreversible) decision problem.\footnote{Note that, when $R=0$, the DM will choose $B$ at time zero, leading to $\overline{q}-\underline{q}=0$.} This shows that
when the exchange cost is zero, the DM tends to make her initial choice
without any exploration. This situation also presents a significant
advantage for selecting product $B$ initially (since both $\underline{q}$ and $%
\overline{q}$ have a small value), as the DM can always switch to product $A$
at no additional cost later on. However, as the cost to return the unknown
product increases, the DM engages in more exploration before reaching her
initial decision in the first exploratory period. 
\end{itemize}

\section{Conclusions and future research directions} \label{sec:conclusion}
We introduced and analyzed a new, integrated decision-making setting that combines sequential decisions, costly information acquisition and distinct informational sources across time regimes. We focused on an E-commerce model in which a decision maker (DM) is presented with the opportunity to choose between a well-known product and a new one, about which she learns through stochastic signals. However, acquiring information about the new product induces information costs that aggregate through time. 

The DM may choose a product but, contrary to the existing models, decision-making continues beyond this time. Specifically, the DM has the optionality to use the product for some time, return it and exchange with the other but at a fee. During this time, she acquires information via a new costly signal, which is more accurate than the initial one.

We formulated the underlying optimal stopping problem which turns out to be a compilation of an ``outer'' and a ``nested'' optimal stopping problem with general payoffs and arbitrary informational costs. For its analysis, we developed a viscosity solution toolkit and analyzed, under rather mild conditions, various cases, performed an extended sensitivity analysis, and studied the limiting behavior of the solution in terms of the various model inputs. We also recovered existing works as special cases and provided some new results for them using viscosity arguments. 

There are five main directions the current work may be extended to. Firstly, we may allow for \textit{multiple information signals} for each decision regime, thus allowing the DM to also choose the type, quality and intensity of the signal he deems most helpful. Secondly, we may allow for \textit{multiple products}, a feature that will induce a more complex product exchange structure that spans to more than two sequential regimes. 
Thirdly, we may consider several DMs who compete for model availability, while sharing common information sources. 
Fourthly, we may develop a more complex fee structure that depends on return duration, product quality, pricing and other factors.
Finally, we may allow for more general information cost functions, beyond the ones that depend only on the belief process. This will naturally lead to higher-dimensional problems, but the nested nature of the involved optimal stopping problems will remain. The authors are currently working on questions related to the above new research directions.

\bibliographystyle{plain}
\bibliography{REF.bib}
\newpage
\appendix
\section*{Electronic Companion}
\label{sec:appendix}

\begin{proof}[Proof of Proposition \ref{prop:constant}]
We establish
monotonicity in the volatility parameter $\sigma $. For $\sigma _{2}>\sigma _{1}$, let $V^{\sigma_{1}}$, $V^{\sigma_{2}} $ be viscosity solutions to 
\begin{align}
\min \left( \rho V^{\sigma_{1}}(q)-\frac{1}{2}\left( \frac{h-l}{\sigma _{1}}\right)
^{2}q^{2}(1-q)^{2}(V^{\sigma_{1}})^{\prime \prime }\left( q\right)
+C(q),V^{\sigma_{1}}(q)-G(q)\right) & =0,  \label{eqn:hjb-sigma1} \\
\min \left( \rho V^{\sigma_{2}}(q)-\frac{1}{2}\left( \frac{h-l}{\sigma _{2}}\right)
^{2}q^{2}(1-q)^{2}(V^{\sigma_{2}})^{\prime \prime }\left( q\right)
+C(q),V^{\sigma_{2}}(q)-G(q)\right) & =0,\label{eqn:hjb-sigma2}
\end{align}%
with $V^{\sigma_{1}}\left( 0\right) =V^{\sigma_{2}}\left( 0\right) =G(0)$ and $V^{\sigma_{1}}\left(
1\right) =V^{\sigma_{2}}\left( 1\right) =G(1)$. We show that $V^{\sigma_{2}}$ is a viscosity
subsolution to \eqref{eqn:hjb-sigma1}. To this end, let $q_{0}\in (0,1)$ and consider a
test function $\varphi \in \mathcal{C}^{2}((0,1))$ such that 
\begin{equation*}
(V^{\sigma_{2}}-\varphi )(q_{0})=\max_{q\in \left( 0,1\right) }(V^{\sigma_{2}}-\varphi )\left(
q\right) =0.
\end{equation*}%
Since $V^{\sigma_{2}}$ is a viscosity subsolution to \eqref{eqn:hjb-sigma2}, $\varphi 
$ must satisfy 
\begin{equation*}
\min \left( \rho \varphi (q_{0})-\frac{1}{2}\left( \frac{h-l}{\sigma _{2}}%
\right) ^{2}q_{0}^{2}(1-q_{0})^{2}\varphi ^{\prime \prime
}(q_{0})+C(q_{0}),\varphi (q_{0})-G(q_{0})\right) \leq 0.
\end{equation*}%
Therefore, it is either the case that $\varphi (q_{0})-G(q_{0})\leq 0$ or 
\begin{equation*}
\varphi (q_{0})-G(q_{0})>0\text{ \ and \ }\rho \varphi (q_{0})-\frac{1}{2}%
\left( \frac{h-l}{\sigma _{2}}\right) ^{2}q_{0}^{2}(1-q_{0})^{2}\varphi
^{\prime \prime }\left( q_{0}\right) +C(q_{0})\leq 0.
\end{equation*}%
The latter inequality implies that $\varphi ^{\prime \prime }(q_{0})>0$ and,
consequently, 
\begin{equation*}
\rho \varphi (q_{0})-\frac{1}{2}\left( \frac{h-l}{\sigma _{1}}\right)
^{2}q_{0}^{2}(1-q_{0})^{2}\varphi ^{\prime \prime }\left( q_{0}\right)
+C(q_{0})<\rho \varphi (q_{0})-\frac{1}{2}\left( \frac{h-l}{\sigma _{2}}%
\right) ^{2}q_{0}^{2}(1-q_{0})^{2}\varphi ^{\prime \prime }\left(
q_{0}\right) +C(q_{0})\leq 0.
\end{equation*}%
Combining the above inequalities and using Lemma \ref{lem:comparison}, we easily
conclude. The rest of the proof follows.
\end{proof}

\begin{proof}[Proof of Proposition \ref{prop:constant:mu}.]
For $\mu
_{2}>\mu _{1}$, let 
\begin{equation*}
G_{1}(q):=\max \Big( \mu _{1},qh+(1-q)l\Big) \text{ \ \ and \ \ }%
G_{2}(q)=:\max \Big( \mu _{2},qh+(1-q)l\Big) .
\end{equation*}%
Then, $G_{2}\left( q\right) \geq G_{1}\left( q\right) \geq G_{2}\left(
q\right) -\mu _{2}+\mu _{1}$. Let $V^{\mu _{1}}$, $V^{\mu _{2}}$ be viscosity
solutions to 
\begin{align}
\min \left( \rho V^{\mu _{1}}(q)-\frac{1}{2}\left( \frac{h-l}{\sigma }\right)
^{2}q^{2}(1-q)^{2}(V^{\mu _{1}})^{\prime \prime }\left( q\right) +C\left( q\right)
,V^{\mu _{1}}(q)-G_{1}(q)\right) & =0,  \label{eqn:hjb-mu1} \\
\min \left( \rho V^{\mu _{2}}(q)-\frac{1}{2}\left( \frac{h-l}{\sigma }\right)
^{2}q^{2}(1-q)^{2}(V^{\mu _{2}})^{\prime \prime }\left( q\right) +C\left( q\right)
,V^{\mu _{2}}(q)-G_{2}(q)\right) & =0. \label{eqn:hjb-mu2}
\end{align}%
Let $q_{0}\in (0,1)$ and $\varphi \in \mathcal{C}^{2}((0,1))$ be such that 
\begin{equation*}
(V^{\mu _{1}}-\varphi )(q_{0})=\max_{q\in \left( 0,1\right) }(V^{\mu _{1}}-\varphi )\left(
q\right) =0.
\end{equation*}%
Since $V^{\mu _{1}}$ is the viscosity solution to \eqref{eqn:hjb-mu1} we have 
\textbf{\ } 
\begin{equation*}
\min \left( \rho \varphi (q_{0})-\frac{1}{2}\left( \frac{h-l}{\sigma }%
\right) ^{2}q_{0}^{2}(1-q_{0})^{2}\varphi ^{\prime \prime }\left(
q_{0}\right) +C(q_{0}),\varphi (q_{0})-G_{1}(q_{0})\right) \leq 0,
\end{equation*}%
and, thus, using that $G_{1}(q)\leq G_{2}(q)$, we also have%
\begin{equation*}
\min \left( \rho \varphi (q_{0})-\frac{1}{2}\left( \frac{h-l}{\sigma }%
\right) ^{2}q_{0}^{2}(1-q_{0})^{2}\varphi ^{\prime \prime }\left(
q_{0}\right) +C(q_{0}),\varphi (q_{0})-G_{2}(q_{0})\right) \leq 0,
\end{equation*}
which implies that $V^{\mu _{1}}$ is a viscosity subsolution to \eqref{eqn:hjb-mu1}%
. Then, by comparison, $V^{\mu _{1}}\leq V^{\mu _{2}}$.

Next, we show that the function $V^{\mu _{2}}-\mu _{2}+\mu _{1}$ is a viscosity
subsolution to \eqref{eqn:hjb-mu1}. For this, we consider a test function $\psi
\in \mathcal{C}^{2}((0,1))$ such that 
\begin{equation*}
(V^{\mu _{2}}-\mu _{2}+\mu _{1}-\psi )(q_{0}):=\max_{q\in \left( 0,1\right)
}(V^{\mu _{2}}-\mu _{2}+\mu _{1}-\psi )\left( q\right) =0.
\end{equation*}%
Then, $\varphi :=\psi +\mu _{2}-\mu _{1}$ satisfies $\varphi \in \mathcal{C}%
^{2}((0,1))$ and 
$(V^{\mu _{2}}-\varphi )(q_{0})=\max_{q\in \left( 0,1\right) }(V^{\mu _{2}}-\varphi
)(q)=0$.
Therefore, 
\begin{equation*}
\min \left( \rho \varphi (q_{0})-\frac{1}{2}\left( \frac{h-l}{\sigma }%
\right) ^{2}q_{0}^{2}(1-q_{0})^{2}\varphi ^{\prime \prime }\left(
q_{0}\right) +C(q_{0}),\varphi (q_{0})-G_{2}(q_{0})\right) \leq 0.
\end{equation*}%
Since $\varphi \geq \psi $, $\varphi ^{\prime \prime }=\psi ^{\prime \prime
} $, and $\varphi -G_{2}=\psi +\mu _{2}-\mu _{1}-G_{2}\geq \psi -G_{1}$, we
deduce that 
\begin{equation*}
\min \left( \rho \psi (q_{0})-\frac{1}{2}\left( \frac{h-l}{\sigma }\right)
^{2}q_{0}^{2}(1-q_{0})^{2}\psi ^{\prime \prime }\left( q_{0}\right)
+C(q_{0}),\psi (q_{0})-G_{1}(q_{0})\right) \leq 0,
\end{equation*}
and, thus, $V^{\mu _{2}}-\mu _{2}+\mu _{1}$ is a viscosity subsolution to %
\eqref{eqn:hjb-mu1}. By comparison, $V^{\mu _{2}}-\mu _{2}+\mu _{1}\leq
V^{\mu _{1}}$.

So far we have shown that 
\begin{equation*}
V^{\mu _{2}}(q)-\mu _{2}+\mu _{1}\leq V^{\mu _{1}}(q)\leq V^{\mu _{2}}(q),\qquad q\in
\lbrack 0,1].
\end{equation*}%
From Theorem \ref{thm:no-regret} we know that there exist two
pairs $\left( \underline{q}_{1},\overline{q}_{1}\right) $ and $\left( 
\underline{q}_{2},\overline{q}_{2}\right) $ such that 
\begin{equation*}
\begin{cases}
V^{\mu _{1}}(q)=\mu _{1}, & 0\leq q\leq \underline{q}_{1} \\ 
V^{\mu _{1}}(q)>G_{1}(q), & \underline{q}_{1}<q<\overline{q}_{1} \\ 
V^{\mu _{1}}(q)=qh+(1-q)l, & \overline{q}_{1}\leq q\leq 1,\quad%
\end{cases}%
\qquad 
\begin{cases}
V^{\mu _{2}}(q)=\mu _{2}, & 0\leq q\leq \underline{q}_{2} \\ 
V^{\mu _{2}}(q)>G_{2}(q), & \underline{q}_{2}<q<\overline{q}_{2} \\ 
V^{\mu _{2}}(q)=qh+(1-q)l, & \overline{q}_{2}\leq q\leq 1.\quad%
\end{cases}%
\end{equation*}%
To compare $\overline{q}_{1}$ and $\overline{q}_{2}$, notice that 
$V^{\mu _{2}}(q)\geq V^{\mu _{1}}(q)\geq qh+(1-q)l$.
For any $q\in \lbrack 0,1]$ such that $V^{\mu _{2}}(q)=qh+(1-q)l$, we also have $%
V^{\mu _{1}}(q)=qh+(1-q)l$. Therefore, 
\begin{equation*}
\lbrack \overline{q}_{2},1]=\Big\{q\in \lbrack
0,1]:V^{\mu _{2}}(q)=qh+(1-q)l\Big\}\subseteq \Big\{q\in \lbrack 0,1]:V^{\mu _{1}}(q)=qh+(1-q)l\Big\}=[%
\overline{q}_{1},1],
\end{equation*}%
which yields $\overline{q}_{2}\geq \overline{q}_{1}$. To compare $\underline{%
q}_{1}$ and $\underline{q}_{2}$, notice that inequality $V^{\mu _{2}}(q)-\mu
_{2}+\mu _{1}\leq V^{\mu _{1}}(q)$ yields 
\begin{equation*}
0\leq V^{\mu _{2}}(q)-\mu _{2}\leq V^{\mu _{1}}(q)-\mu _{1}.
\end{equation*}%
For any $q\in \lbrack 0,1]$ such that $V^{\mu _{1}}(q)=\mu _{1}$, we also have $%
V^{\mu _{2}}(q)=\mu _{2}$. Therefore, 
\begin{equation*}
\lbrack 0,\underline{q}_{1}]=\{q\in \lbrack 0,1]:V^{\mu _{1}}(q)=\mu
_{1}\}\subseteq \{q\in \lbrack 0,1]:V^{\mu _{2}}(q)=\mu _{2}\}=[0,\underline{q}%
_{2}],
\end{equation*}%
which yields $\underline{q}_{2}\geq \underline{q}_{1}$. In conclusion, for $%
\mu _{2}>\mu _{1}$, we have $\underline{q}_{2}\geq \underline{q}_{1}$ and $%
\overline{q}_{2}\geq \overline{q}_{1}$.
\end{proof}

\begin{proof}[Proof of Proposition \ref{prop:constant:limit}]
    We only
analyze the case $\sigma \uparrow \infty $ as the other two follow
similarly. To this end, we recall that $V$ is the viscosity solution to (%
\ref{OP}), rewritten for convenience, 
\begin{equation*}
\min \left( \rho V(q)-\frac{1}{2}\left( \frac{h-l}{\sigma }\right)
^{2}q^{2}(1-q)^{2}V^{\prime \prime }\left( q\right) +C(q),V(q)-G(q)\right)
=0,
\end{equation*}%
with $G(q)=\max \left( \mu ,qh+\left( 1-q\right) l\right) $. We now construct a
suitable $\mathcal{C}^{1}((0,1))$ supersolution, introducing 
\begin{eqnarray*}
U(q):=
\begin{cases}
\mu,  & 0\leq q\leq r \\ 
\mu +M(q-r)^{2}, & r\leq q\leq p \\ 
qh+(1-q)l, & p\leq q\leq 1,%
\end{cases}%
\end{eqnarray*}
for some $r\in (0,\hat{p})$, $p\in (\hat{p},1)$ and $M>0$. To have $U\in 
\mathcal{C}^{1}(0,1)$, we need 
\begin{eqnarray}
&&U(p) =\mu +M(p-r)^{2}=ph+(1-p)l,  \label{eqn:solving-r-p-1} \\
&&U^{\prime }(p) =2M(p-r)=h-l.\label{eqn:solving-r-p-2}
\end{eqnarray}%
By construction, $U\left( q\right) =G\left( q\right) $ on $%
[0,r]\cup \lbrack p,1]$. Therefore, it suffices to verify that 
\begin{equation*}
\rho U(q)-\frac{1}{2}M\left( \frac{h-l}{\sigma }\right)
^{2}q^{2}(1-q)^{2}+C(q)>0,\text{ \ }q\in \lbrack r,p],
\end{equation*}%
which follows for large $\sigma .$\textbf{\ }In turn, we use\textbf{\ }Lemma %
\ref{lem:comparison} to deduce that $r\leq \underline{q}\leq \overline{q}%
\leq p$.

We finish the proof by showing that the points $r$ and $p$ can be
arbitrarily close to $\hat{p}$ (cf. \eqref{eqn:k}) if we choose $M$ sufficiently large. Indeed,
applying \eqref{eqn:solving-r-p-2} to \eqref{eqn:solving-r-p-1} gives 
\begin{equation*}
\mu +\frac{(h-l)^{2}}{4M}=ph+(1-p)l=\mu +(h-l)(p-\hat{p}),
\end{equation*}%
which yields $p=\hat{p}+\dfrac{h-l}{4M}.$ Using %
\eqref{eqn:solving-r-p-2}, we deduce that $r=\hat{p}-\dfrac{h-l}{4M}$. Hence 
$p\rightarrow \hat{p}$ and $r\rightarrow \hat{p},$ as $M\rightarrow \infty $%
. 
\end{proof}

\begin{proof}[Proof of Proposition \ref{prop:constant-2}.]
We first
prove \ref{prop-2:i}. Note that $\hat{p}=\frac{\mu -l}{h-l}%
\downarrow 0$ when $l\uparrow \mu $. In this case, $\underline{q}\rightarrow
0$ as $\hat{p}\rightarrow 0$.
Next, we construct a suitable convex supersolution $U\in \mathcal{C}([0,1])$
to (\ref{OP}). For this, let 
\begin{equation*}
U(q):=%
\begin{cases}
\mu, & q=0, \\ 
qh+(1-q)l+M(p-q)^{2}, & 0<q\leq p, \\ 
qh+(1-q)l, & p<q\leq 1,%
\end{cases}%
\end{equation*}%
for some $M>0$ and $p\in (\hat{p},1)$ to be determined in the sequel.

Firstly, note that continuity at $q=0$ requires $M$ and $p$ to satisfy $%
l+p^{2}M=\mu$. Thus, 
\begin{equation*}
p=\sqrt{\frac{\mu -l}{M}},
\end{equation*}%
where $M$ will be chosen independently of $l$.

Next, we verify that $U$ is convex. Since $U^{\prime \prime }(q)=2M>0$, $%
q\in (0,p),$ and $U$ is affine on $[p,1]$, it suffices to compute the
left derivative of $U$ at $p$. We have 
\begin{equation*}
U_{-}^{\prime }(p)=(h-l)+2M(p-p)=h-l.
\end{equation*}%
Therefore, $U$ is convex in $[0,1]$ and piecewise smooth. Note also that $U' (0) = h - l - 2 M p > h - \mu - 2 M p$, which is positive for $p$ sufficiently small as $l$ gets sufficiently close to $\mu$. Therefore, $U$ is strictly increasing in $[0, 1]$.

We now verify that $U$ is a supersolution. Notice that $U\left( q\right)
=G\left( q\right) ,$ for $q=0$ and $q\in \lbrack p,1]$. By monotonicity, $U(q)\ge U(0)=\mu$. From the definition, we also have that $U(q)\ge qh+(1-q)l$ and, thus, $U \ge G$ in $[0, 1]$. To show that $U$ is
a supersolution, it suffices to establish that, for $q\in (0,p)$, it holds
that 
\begin{equation*}
f(q):=\rho U(q)-\frac{1}{2}\left( \frac{h-l}{\sigma }\right)
^{2}q^{2}(1-q)^{2} U''(q)+C(q)\geq 0.
\end{equation*}%
We have that 
\begin{equation*}
f(q)\geq \rho \mu +C_1-M\frac{h^{2}}{16\sigma ^{2}},
\end{equation*}%
since $U(q)\geq\mu $ and $q^{2}(1-q)^{2}\leq \frac{1}{16},$ for $q\in
\lbrack 0,1]$, and where we also need that $C (q) \ge C _1$ (cf. Assumption \ref{ass:CI}).

By choosing $M=\frac{16(\rho
\mu +C_1)\sigma ^{2}}{h^{2}}$, we have $f(q)\geq 0$. Therefore,
choosing also $p=\sqrt{\frac{\mu -l}{M}}$, we have that $U$ is a supersolution and that $U(q)>G(q),$ $%
q\in (0,p)$. Then, by the comparison principle for continuous viscosity solutions, we obtain that $V(q)\leq
U(q),$ $q\in (0,p)$ and, hence, $p\geq \overline{q}$. As $l\rightarrow \mu $%
, we have $\underline{q}\rightarrow 0$ and $\overline{q}\rightarrow 0$, since
${\hat{p}}\rightarrow 0$ and $p\rightarrow 0$.

\bigskip

To show \ref{prop-2:ii}, we construct a convex subsolution $%
\overline{U}\in \mathcal{C}([0,1])$ with $\overline{U}=\max \left( \mu
,U\right) $ for a suitable function $U.$ To this end, for some positive
constants $m,M>0,r\in \left( 0,\frac{1}{2}\right] $ and $p\in \left( \frac{1%
}{2},1\right) $ to be determined, we define 
\begin{equation*}
U(q):=\int_{q}^{1}\Phi (t)dt+qh+(1-q)l,
\end{equation*}%
with
\begin{equation*}
U^{\prime \prime }(q)=\varphi (q)=%
\begin{cases}
M, & 0\leq q\leq r, \\ 
m, & r<q\leq p, \\ 
0, & p<q\leq 1,%
\end{cases}%
\quad \Phi (q)=\int_{q}^{1}\varphi (t)dt=%
\begin{cases}
m(p-r)+M(r-q), & 0\leq q\leq r, \\ 
m(p-q), & r<q\leq p, \\ 
0, & p<q\leq 1.%
\end{cases}%
\end{equation*}%
For the reader's convenience, we graph function $U$ below.
\begin{figure}[ht]
     \centering     \includegraphics[width=15cm]{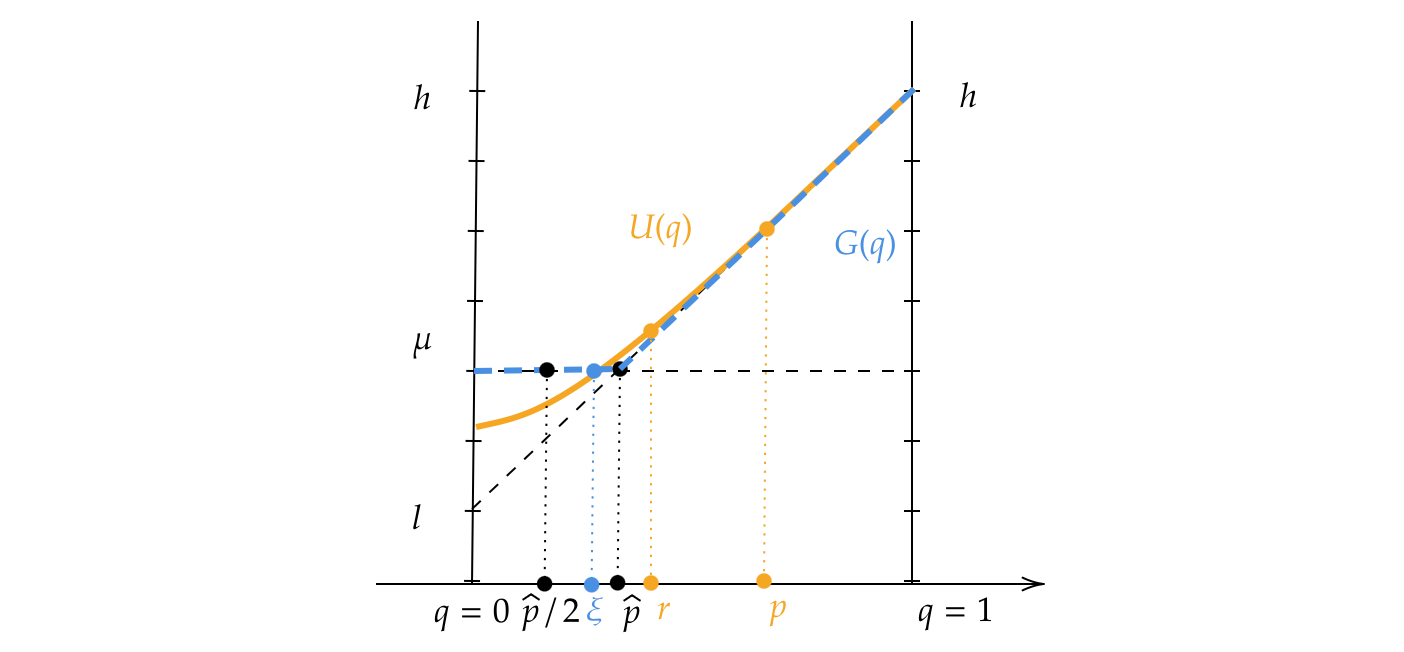}
     \caption{Illustration of $U(q)$ (orange solid line) and $G(q) = \max\big(\mu,q h +  (1-q) l\big)$ (blue dotted line).}
     \label{fig:h-sub}
 \end{figure}

We choose $m,M,r$ independently of $h$. We claim that by a proper choice of $%
m,M,p,r$, $U$ will satisfy the following properties. %
Firstly, $U$ is convex since $U^{\prime }(q)=-\Phi (q)+h-l$ is increasing in 
$q$. Moreover, $\Phi $ is Lipschitz, and $U$ is $\mathcal{C}([0,1])$ since $%
\varphi $ is bounded by $\max \left( m,M\right) $. We, also, have
\begin{equation*}
U^{\prime }(0)=-\Phi (0)+h-l=-m(p-r)-Mr+h-l>-M+h-l,
\end{equation*}%
where $M$ (as mentioned above) will be chosen independently of $h$. We also have $U^{\prime }(0)>0,$
for $h$ is sufficiently large.

Next, we determine $\xi $ at which $U(\xi )=\mu $. By direct calculations, and choosing $m=\frac{1}{2}(\mu -l)$ and $Mr^{2}=\frac{1}{2}(\mu -l)$, we have
\begin{align}
U\left( \frac{\hat{p}}{2}\right) & =\int_{\frac{\hat{p}}{2}}^{1}\Phi
(t)dt+h\frac{\hat{p}}{2}+l\left( 1-\frac{\hat{p}}{2}\right) \\
& =\int_{\frac{\hat{p}}{2}}^{1}\Phi (t)dt+h\hat{p}+l\left( 1-\hat{p}\right) -\frac{\hat{p}}{2}(h-l) \\
& \leq \int_{0}^{1}\Phi (t)dt+\mu -\frac{1}{2}(\mu -l) \\
& =mr(p-r)+\frac{M}{2}r^{2}+\frac{m}{2}(p-r)^{2}+\mu -\frac{1}{2}(\mu -l) \\
& <\frac{1}{2}m+\frac{1}{2}Mr^{2}+\mu -\frac{1}{2}(\mu -l)<\mu.
\end{align}%
On the other hand, since $U(\hat{p})>\hat{p}h+(1-\hat{p})l=\mu $ and $U$ is monotonically
increasing, there exists a unique $\xi \in (\frac{\hat{p}}{2},\hat{p})$
such that $U(\xi )=\mu $. 
By the monotonicity of $U$, we also have that $U(0)<U(\xi )=\mu $. Moreover, $$%
U(q)\leq U(0)+q(h-l)\leq \mu +q(h-l),$$ since $U^{\prime }(q)=-\Phi
(q)+h-l\leq h-l$, $q\in \lbrack 0,1]$.

We are now ready to show that $\overline{U}=\max \left( \mu ,U\right) $ is
indeed a viscosity subsolution. To this end, notice that $\overline{U}(q)=G(q)$, $q\in
\lbrack 0,\xi ]\cup \lbrack p,1]$. To show that $\overline{U}$ is a
subsolution, it suffices to show that, for $q\in (\xi ,p)$, we have 
\begin{equation*}
\rho U(q)-\frac{1}{2}\left( \frac{h-l}{\sigma }\right)
^{2}q^{2}(1-q)^{2}U_{-}^{\prime \prime }\left( q\right) +C(q)<0.
\end{equation*}%

To this end, for each $q\in (\xi ,r]\subset (\frac{\hat{p}}{2},r]$, since $%
U_{-}^{\prime \prime }(q)=M$ for $q<r\leq \frac{1}{2}$ and $U(q)\leq \mu
+(h-l)q$, we have 
\begin{align*}
\rho U(q)-\frac{1}{2}\left( \frac{h-l}{\sigma }\right)
^{2}q^{2}(1-q)^{2}U_{-}^{\prime \prime }\left( q\right) +C(q)& \leq \rho \mu
+C_2+\rho (h-l)q-\frac{M}{8\sigma ^{2}}{(h-l)^{2}q^{2}}  \notag \\
& \leq \rho \mu +C_2+\frac{4\rho ^{2}\sigma ^{2}}{M}-\frac{M}{16\sigma
^{2}}{(h-l)^{2}q^{2}} \\
& \leq \rho \mu +C_2+\frac{4\rho ^{2}\sigma ^{2}}{M}-\frac{M}{16\sigma
^{2}}{(h-l)^{2}\left( \frac{\hat{p}}{2}\right) ^{2}} \\
& =\rho \mu +C_2+\frac{4\rho ^{2}\sigma ^{2}}{M}-\frac{M}{64\sigma ^{2}}%
(\mu -l)^{2},
\end{align*}%
where we need that $C (q) \le C _2$ (cf. Assumption \ref{ass:CI}).

Next, we choose $M$ sufficiently
large such that the last quantity above is negative. For instance, it suffices to take $M=\max \left( \frac{128\sigma ^{2}(\rho \mu +C_2)}{(\mu
-l)^{2}},\frac{32\rho \sigma ^{2}}{\mu -l},2(\mu -l)\right) $. 
Having fixed such $M$, we choose $r = \sqrt{\frac{\mu-l}{2M}}$ which ensures that $r \le \tfrac12$.
We note that the
choices of $m$, $M$, $r$ are all independent\textbf{\ }of the value of $h$.

\textbf{\ }For $q\in (r,p)$, since $U^{\prime \prime }(q)=m$ and $%
U(q)\leq U(1)=h$, by choosing 
\begin{equation*}
p:=1-\sqrt{\frac{2\sigma ^{2}(\rho h+C_1)}{m(h-l)^{2}r^{2}}},
\end{equation*}%
we have 
\begin{equation*}
\rho U(q)-\frac{1}{2}\left( \frac{h-l}{\sigma }\right)
^{2}q^{2}(1-q)^{2}U^{\prime \prime }\left( q\right) +C(q)<\rho h-\frac{1}{2}%
\left( \frac{h-l}{\sigma }\right) ^{2}mr^{2}(1-p)^{2}+C_2=0.
\end{equation*}%
Note that $p\rightarrow 1$ as $h\rightarrow \infty .$

In conclusion, by choosing $m=\frac{1}{2}(\mu -l)$, $M=\max \left( \frac{%
128\sigma ^{2}(\rho \mu +C_2)}{(\mu -l)^{2}},\frac{32\rho \sigma ^{2}}{%
\mu -l},2(\mu -l)\right) $,
$r=\sqrt{\frac{\mu -l}{2M}}$, and $p=1-\sqrt{\frac{2\sigma ^{2}(\rho h+%
C_1)}{m(h-l)^{2}r^{2}}}$, we establish that $\overline{U}$ is a viscosity
subsolution with $\overline{U}(q)>G(q),$ $q\in (\xi ,p)$. By the comparison principle for continuous viscosity solutions, we
deduce that $V(q)\geq \overline{U}(q)>G(q),$ for $q\in (\xi ,p),$ and hence $%
\underline{q}\leq \xi <p\leq \overline{q}$. As $h\rightarrow \infty $,
because $\xi <\hat{p}\rightarrow 0$ and $p\rightarrow 1$, we obtain that $%
\underline{q}\rightarrow 0$ and $\overline{q}\rightarrow 1$. 
\end{proof}

\end{document}